\documentclass[11pt]{amsart}
\usepackage{mathtools, amssymb, amsfonts, amsthm, enumitem, array}
\usepackage{epsfig, psfrag, color, multicol, graphicx, tikz, longtable}
\usepackage{amsfonts}
\usepackage{epstopdf}

\newtheorem{thm}{Theorem}[section]

\newtheorem{cor}[thm]{Corollary}

\newtheorem{question}[thm]{Question}

\theoremstyle{plain}
\newtheorem{theo}[thm]{Theorem}
\newtheorem{lem}[thm]{Lemma}

\theoremstyle{definition}

\newtheorem{eg}[thm]{Example}
\newtheorem{rem}[thm]{Remark}

\numberwithin{equation}{section}

\def\bu{\bullet}

\def\sq{\square}

\def\zz{\mathbb Z}
\def\nn{\mathbb N}

\def\rr{\mathbb R}

\def\ov{\overline}

\def\Ga{\Gamma}

\def\De{\Delta}

\def\ga{\gamma}

\def\ep{\ve}
\def\al{\alpha}
\def\be{\beta}

\def\ve{\varepsilon}

\def\vk{s}

\def\T{\mathbf{T}}

\def\ssu{\subset}

\def\wt{\widetilde}
\def\<{\langle}
\def\>{\rangle}

\def\Z{ {\text {\rm Z} } }

\def\Q{{\text {\rm Q} } }

\def\0{{\mathbf 0}}

\def\NN{{\mathbb N}}

\def\.{\hskip.06cm}
\def\ts{\hskip.03cm}

\def\lra{\leftrightarrow}

\def\bz{{\textbf{z}}}
\def\bx{{\textbf{x}}}

\def\by{{\textbf{y}}}

\def\poly{\textup{\textsf{P}}}

\def\SP{{\textup{\textsf{\#P}}}}
\def\SigmaP{\Sigma^{\poly}}
\def\PiP{\Pi^{\poly}}

\def\Z{\mathbb{Z}}
\def\R{\mathbb{R}}

\def\Q{\mathbb{Q}}

\def\T{\mathcal{T}}

\newcommand{\cj}[1]{\overline{#1}}
\newcommand{\n}{\cj{n}}
\renewcommand{\b}{\cj{b}}
\def\albar{\cj{\alpha}}
\def\nubar{\cj{\nu}}

\newcommand{\x}{\mathbf{x}}
\renewcommand{\t}{\mathbf{t}}
\renewcommand{\u}{\mathbf{u}}
\def\v{\mathbf{v}}

\def\w{\mathbf{w}}

\newcommand{\y}{\mathbf{y}}
\newcommand{\z}{\mathbf{z}}

\newcommand{\floor}[1]{\lfloor#1\rfloor}

\newcommand{\polyin}[1]{\textup{poly}(#1)}

\newcommand{\ex}{\exists\ts}
\renewcommand{\for}{\forall\ts}

\def\nin{\noindent}

\def\NP{{\textup{\textsf{NP}}}}
\def\PP{{\textup{\textsf{P}}}}

\def\sharpP{\textup{\textsf{\#P}}}

\def\IPP{\text{\rm (IP)}}
\def\PIPP{\text{\rm (PIP)}}
\def\GIP{\text{\rm (GIP)}}
\def\sharpGIP{\text{\rm ($\#$GIP)}}
\def\SPAo{\text{\rm (Short-PA$_1$)}}
\def\SPAd{\text{\rm (Short-PA$_2$)}}
\def\SPAt{\text{\rm (Short-PA$_3$)}}
\def\SPAm{\text{\rm (Short-PA$_m$)}}
\def\sharpSPAt{\text{\rm ($\#$Short-PA$_3$)}}

\def\ovv{\overrightarrow}
\def\seg{}
\def\AP{\textup{AP}}

\newcommand{\problem}[1]{\textsc{#1}}

\newcommand{\problemdef}[3]{
\bigskip
\begin{tabular}{p{0.1\textwidth} p{0.8\textwidth}}
\multicolumn{2}{l}{\problem{#1}}\\
\textbf{Input:} & {#2} \\
\textbf{Decide:} & {#3}
\end{tabular}
\bigskip
}

\renewcommand{\mod}[1]{
\;\, (\textup{mod} \; #1)
}

\def\lra{\leftrightarrow}
\def\KPT{\textnormal{KPT}}

\def\PIP{\textnormal{PIP}}
\def\cpl{\ts\backslash\ts}
\def\cal{\mathcal}

\title{Short Presburger arithmetic is hard$\phantom{}^\dagger$}

\author[Danny Nguyen \and Igor Pak]{Danny Nguyen$^{\star}$ \and Igor~Pak$^{\star}$}

\thanks{\thinspace${\hspace{-.45ex}}^\dagger$Extended abstract will appear in \emph{Proceedings of the 58th Annual Symposium on Foundations of Computer Science (FOCS 2017)}.
}

\thanks{\thinspace ${\hspace{-.45ex}}^\star$Department of Mathematics,
UCLA, Los Angeles, CA, 90095.
\hskip.06cm
Email:
\hskip.06cm
\texttt{\{ldnguyen,\ts{pak}\}@math.ucla.edu}}

\thanks{
\today}

\begin{document}
\maketitle


\begin{abstract}
We study the computational complexity of short sentences in Presburger arithmetic (\textsc{Short-PA}).
Here by ``short''  we mean sentences with a bounded number of variables, quantifiers,
inequalities and Boolean operations; the input consists only of the integer coefficients involved
in the linear inequalities.  We prove that satisfiability of \textsc{Short-PA} sentences
with $m+2$ alternating quantifiers is $\SigmaP_{m}$-complete or $\PiP_{m}$-complete,
when the first quantifier is $\exists$ or $\forall$, respectively.  Counting versions
and restricted systems are also analyzed.  Further application are given to
hardness of two natural problems in Integer Optimization.
\end{abstract}

\vskip1.5cm

\section{Introduction}

\subsection{Outline of the results}
We consider \emph{short Presburger sentences}, 
defined as follows:
$$
\SPAm
\qquad \ex\x_{1} \;\; \forall \ts\x_{2} \; \dots \;
\forall/\exists \ts \x_{m} \,:\, \Phi\bigl(\x_{1}, \dots, \x_{m}\bigr),
$$
where the quantifiers alternate, the variables $\x_i \in \zz^{n_i}$
have fixed dimensions $\ov n=(n_1,\ldots,n_m)$, and $\Phi(\x_1,\ldots,\x_m)$
is a fixed Boolean combination of integer linear systems of fixed lengths (numbers of inequalities):
$$(\ast) \qquad
A_1\ts \x_1 \. + \. \ldots \. + \. A_k\ts \x_m \, \le \, \ov b\ts.
$$
In other words, everything is fixed in~$\SPAm$,  except for the entries
of the matrices~$A_i$ and of the vectors~$\ov b$ in $(\ast)$.
We also call $\Phi$ a \emph{short Presburger expression}.

The feasibility of short Presburger sentences is a well known open problem
which we resolve in this paper.  Connected to both Integer Programming and
Computational Logic, it was called a ``fundamental question'' by Barvinok
in a recent survey~\cite{B4}.
Many precursors to $\SPAm$ are well known, including
\emph{Integer Linear Programming}:
$$
\IPP
\qquad
\exists \ts\bx \, : A\ts\x \le \b\ts,
$$
and \emph{Parametric Integer Programming}:
$$
\PIPP
\qquad
\forall \ts\by \in Q \;\; \exists \ts\bx \, : \, A\ts\bx \. + \. B\ts\by \, \le \, \ov b\ts,
$$
where $Q$ is a convex polyhedron given by $\ts K \ts\by \le \ov u$.
In both cases, the problems were shown to be in~$\PP$, by Lenstra in~1982
and Kannan in~1990, respectively (Theorem~\ref{th:Kannan}).
Traditionally, the lengths of the systems in both $\IPP$ and $\PIPP$ are not restricted.
However, it is known that they both can be reduced to the case of a bounded length system (c.f. Sec.~8.1~\cite{KannanNPC}).

Our main result is a complete solution of the problem.  We show that for a fixed
$m \ge 3$, deciding~$\SPAm$ is $\SigmaP_{m-2}$-complete (Theorem~\ref{th:PH}).
This disproves\footnote{Assuming the polynomial hierarchy does not collapse.}
a conjecture by Woods~\cite[$\S$5.3]{W1} (see also~\cite{W2}),
which claims that decision is in $\poly$.

Let us emphasize that until this work even the following special case remained open:
$$
\GIP
\qquad
\exists \ts\bz \in R \;\; \forall \ts\by \in Q \;\; \exists \ts\bx \, : \, A\ts\bx \. + \. B\ts\by \. + \. C\ts\bz \, \le \, \ov b\ts,
$$
where $Q$ and $R$ are convex polyhedra given by $\ts K \ts\by \le \ov u$ and $\ts L \ts\bz \le \ov v$,
respectively.  We also show that $\GIP$ is $\NP$-complete (Theorem~\ref{cor:system1}).
This resolves an open problem by Kannan~\cite{K2}.

Our reduction is parsimonious and also proves that the corresponding counting problem is
$\SP$-complete:
$$
\sharpGIP
\qquad
\# \, \big\{ \ts\bz \in R \; : \; \forall \ts\by \in Q \;\; \exists
\ts\bx \;\; A\ts\bx \. + \. B\ts\by \. + \. C\ts\bz \, \le \, \ov b\ts  \big\}.
$$

There is a natural geometric way to view these problems.  Problem~$\IPP$
asks whether a given rational polyhedron $P \ssu \rr^d$ contains an integer point.
Problem~$\PIPP$ asks whether the projection of~$P$ contains all integer points
in some polyhedron~$Q$.  Finally, problem~$\GIP$ asks whether there is an
$R$-slice of a polyhedron $P$ for which the projection contains all integer points
in some polyhedron~$Q$.

\subsection{Precise statements}
For $m=3$ alternating quantifiers, we have the first hard instance of $\SPAm$~:
$$
\SPAt
\qquad
\exists \ts\bz  \;\; \forall \ts\by \;\; \exists \ts\bx \, : \, \Phi(\x,\y,\z).
$$
Here $\Phi$ is a short Presburger expression in $\x$, $\y$ and $\z$.
We can also define the counting problem
$$
\sharpSPAt
\qquad
\# \big\{ \ts\z \; : \; \for \y \;\; \ex \x \;\; \Phi(\x,\y,\z) \ts \big\}.
$$

\smallskip

\begin{theo}\label{th:PR3hard}
Deciding $\SPAt$ is $\NP$-complete, even for a short
Presburger expression $\Phi$ of at most  $10$ inequalities in $5$ variables
$z \in \zz$, $\y \in \zz^{2}$, $\x \in \zz^{2}$.
Similarly, computing $\sharpSPAt$ in this case is $\sharpP$-complete.
\end{theo}

For systems of inequalities, we also get:

\begin{theo}\label{cor:system1}
Deciding $\GIP$ is $\NP$-complete, even for a system $A\ts\x + B\ts\y + Cz  \le \ov b\ts$
of at most $24$ inequalities in $9$ variables  $z \in \zz$,
$\y \in \zz^{2}$, $\x \in \zz^{6}$, when $R$ is an interval and $Q$ is a triangle.
Similarly, computing $\sharpGIP$ in this case is $\sharpP$-complete.
\end{theo}

The third dimension $\x \in \zz^{6}$ in the theorem can be lowered to $\x \in \zz^{3}$
at the cost of increasing the length of the linear system:

\begin{theo}\label{cor:system2}
Deciding $\GIP$ is $\NP$-complete, even for a system
$A\ts\x + B\ts\y + Cz  \le \ov b\ts$ of at most $8400$
inequalities in $6$ variables  $z \in \zz$, $\y \in \zz^{2}$, $\x \in \zz^{3}$,
when $R$ is an interval and $Q$ is a triangle.
Similarly, computing $\sharpGIP$ in this case is $\sharpP$-complete.
\end{theo}

This substantially strengthens our earlier result~\cite{KannanNPC}, which considers $\GIP$ with a ``long system'', i.e., a system arbitrarily many inequalities:

\begin{theo}[\cite{KannanNPC}]\label{th:NP-Kannan}
Deciding $\GIP$ is $\NP$-complete, for a system $A\ts\x + B\ts\y + Cz  \le \ov b\ts$
of unbounded length in $6$ variables $z \in \zz$, $\y \in \zz^{2}$, $\x \in \zz^{3}$.
\end{theo}

At the time of proving Theorem~\ref{th:NP-Kannan}, we thought it would be the strongest negative result (see Section~\ref{sec:KPT_intro} below).
Nevertheless, the new results in theorems~\ref{th:PR3hard},~\ref{cor:system1} and~\ref{cor:system2} say that at the level of three quantifiers, both Integer Programming
and Presburger Arithmetic quickly saturate to a high level of complexity,
even when all para\-meters are bounded.

The decision part of Theorem~\ref{th:PR3hard}
can naturally be generalized to short Presburger sentences of more than $3$ quantifiers:

\begin{theo}[Main result]\label{th:PH}
Fix $m \ge 1$.
Let $Q_{1},\dots,Q_{m+2} \in \{\for,\ex\}$ be $m+2$ alternating quantifiers with $Q_{1} = \ex$.
Deciding short Presburger sentences of the form
\begin{equation*}
Q_{1}\z_{1}  \;\; \dots \;\; Q_{m+1} \z_{m+1} \;\; Q_{m+2}\z_{m+2} \;\; : \;\; \Phi(\z_{1},\dots,\z_{m+2})
\end{equation*}
is $\SigmaP_{m}$-complete.  Similarly, when $Q_{1} = \forall$, deciding short Presburger sentences
as above is $\PiP_{m}$-complete.
Here $\Phi$ is a short Presburger expression of at most $10m$ inequalities
in $4m+1$ variables $\z_{1} \in \zz$, $\z_{2}, \z_{m+2}  \in \zz^{2}$,
and $\z_{3},\dots,\z_{m+1}\in \zz^{4}$.
\end{theo}

The proof of the above results uses a chain of reductions.  We start with
the $\textsc{AP-COVER}$ problem on covering intervals with arithmetic progressions.
This problem is $\NP$-compete by a result of Stockmeyer and Meyer~\cite{MS}
(see Section~\ref{sec:AP}).  The arithmetic progressions are encoded via continued fractions by
a single rational number~$p/q$.  We use the plane geometry of
continued fractions and ``lift'' the construction to a Boolean
combination of polyhedra in dimension~5, proving Theorem~\ref{th:PR3hard}.
We then ``lift'' the construction further to convex polytopes $Q_1 \ssu \rr^9$ and
$Q_2 \ssu \rr^6$, which give proofs of theorems~\ref{cor:system1} and~\ref{cor:system2},
respectively.  While both constructions are explicit, the first construction
gives a description of~$Q_1$ by its 24~facets, while the second gives a description
of~$Q_2$ by its 40 vertices; the bound of~8400 facets then comes from McMullen's
Upper bound theorem (Theorem~\ref{lem:upper_bound}).  Finally, we generalize
the problem  $\textsc{AP-COVER}$ and the chain of reductions to $m\ge 3$ quantifiers.


\subsection{Applications in integer optimization}\label{ss:application}
The first application of our construction is the following hardness result
on the \emph{bilevel optimization} of a quadratic function over integer
points in a polytope.

\begin{theo}\label{th:minmax}
	Given a rational interval $J \subset \rr$, a rational polytope $W \subset \rr^{5}$ and a quadratic rational polynomial
$h : \rr^{6} \to \rr$, computing:
	\begin{equation}\label{eq:minmax}
	\max_{z \in J \cap \zz} \quad \min_{\w \in W \cap \zz^{5}} \quad h(z,\w)
	\end{equation}
	is $\NP$-hard. This holds even when $W$ has at most $18$ facets.
\end{theo}

The polytope $W$ can be given either by its vertices or by its facets, as the
theorem holds in both cases.

\smallskip

The second application is to the hardness of the \emph{Pareto optima}.
Assume we are given polytope $Q \subset \rr^{n}$, and $k$ functions
$f_{1},\dots,f_{k}: \rr^{n} \to \rr$ restricted to the domain $Q \cap \zz^{n}$.
For a point $\x \in Q \cap \zz^{n}$, the corresponding outcome vector
$\y = (f_{1}(\x), \dots, f_{k}(\x))$ is called a \emph{Pareto minimum}, if there is no other point
$\wt \x \in Q \cap \zz^{n}$ and $\wt \y = (f_{1}(\wt \x), \dots, f_{k}(\wt \x))$,
such that $\wt \y \le \y$ coordinate-wise and $\wt \y \ne \y$.
The goal is to minimize the value of an \emph{objective function} $g:\rr^{k} \to \rr$
over all Pareto minima $\y$ of $(f_{1},\dots,f_{k})$ on~$Q$.

\begin{theo}\label{th:Pareto}
Given a rational polytope $Q \subset \rr^{6}$, two rational linear functions
$f_{1},f_{2}: \rr^{6} \to \rr$, a rational quadratic polynomial
$f_{3}:\rr^{6} \to \rr$, and rational linear objective function
$g:\rr^{3} \to \rr$, computing the minimum of $g$ over the Pareto minima of
$(f_{1},f_{2},f_{3})$ on $Q$ is $\NP$-hard. Moreover, the corresponding
$1/2$-approximation problem is also $\NP$-hard.
This holds even when $Q$ has at most $38$ facets.
\end{theo}

Again, the polytope $Q$ can be given either by its vertices or by its facets.
Here by $\ve$-approximation we mean approximation up to a multiplicative factor
of~$\ve$.

\smallskip

We prove both theorems in Section~\ref{sec:optimization}.
See also $\S$\ref{ss:finrem-minmax} and $\S$\ref{ss:finrem-Pareto}
for some background and open problems.


\subsection{Historical overview}
\emph{Presburger Arithmetic} was introduced by Presburger in~\cite{Pre},
where he proved it is a decidable theory. The general theory allows unbounded
numbers of quantifiers, variables and Boolean operations.
 A quantifier elimination (deterministic) algorithm was given by Cooper~\cite{C},
 and was shown to be triply exponential by Oppen~\cite{O} (see also~\cite{RL}).
A nondeterministic doubly exponential complexity lower bound was obtained
by Fischer and Rabin~\cite{FR} for the general theory.  This pioneering
result was further refined to
a triply exponential deterministic lower bound (with unary output) in~\cite{Wei},
and a simply exponential nondeterministic lower bound for a bounded
number of quantifier alternations~\cite{Fur} (see also~\cite{Sca}).
Of course, in all these cases the number of variables is unbounded.

In~\cite{S}, Sch\"{o}ning proved \NP-completeness for two quantifiers
$\ts\exists y \ts \forall x\ts:\ts\Phi(x,y)$, where $x,y\in \zz$
and $\ts\Phi(x,y)\ts$ is a Presburger expression in $2$ variables,
i.e., a Boolean combination of arbitrarily many inequalities in $x,y$.
This improved on an earlier result by Gr\"{a}del, who also established
that similar sentences with $m+1$ alternating quantifiers and a
bounded number of variables are complete for the $m$-th level in
the Polynomial Hierarchy~\cite{G}.  Roughly speaking, one can view
our results as variations on Gr\"{a}del's result, where we trade
boundedness of~$\Phi$ for an extra quantifier.

Let us emphasize that when the number of variables is unbounded, even the most
simple systems~$\IPP$ become $\NP$-complete. The examples include the
\textsc{KNAPSACK}, one of the oldest $\NP$-complete problems~\cite{GJ}.
Note also that even when matrix~$A$ has at most
two nonzero entries in each row, the problem remains $\NP$-complete~\cite{Lag}.


In a positive direction, the progress has been limited.  The first
breakthrough was made by Lenstra~\cite{L}  (see also~\cite{Schrijver}),
who showed that $\IPP$ can be solved
in polynomial time in a fixed dimension (see also~\cite{E1} for better bounds).
Combined with a reduction by Scarpellini~\cite{Sca}, this implies that deciding
$\SPAo$ is in~$\PP$.

The next breakthrough was made by Kannan~\cite{K1} (see also~\cite{K2}),
who showed that $\PIPP$ in fixed dimensions is in~$\PP$, even if the number~$s$
of inequalities is unbounded, i.e. the matrices $A$ and $B$ can
be ``long''.  This was a motivation for our earlier Theorem~\ref{th:NP-Kannan} from~\cite{KannanNPC},
which ruled out ``long'' systems for~$\GIP$.

\begin{theo}[Kannan]\label{th:Kannan}
Fix $n_1,n_2$.  The formula $\PIPP$ in variables
$\x \in \zz^{n_{1}}$, $\y \in \zz^{n_{2}}$ with $s$ inequalities can be
decided in polynomial time, where $s$ is part of the input.
\end{theo}

Kannan's Theorem was further strengthened by Eisenbrand and Shmonin~\cite{ES}
(see~$\S$\ref{ss:kannan-ES}).  All of these greatly contrast with the
above hardness results by Sch\"{o}ning and Gr\"{a}del,
because here only conjunctions of inequalities are allowed.

The corresponding counting problems have also been studied with great success.
First, Barvinok~\cite{B1} showed that integer points in a convex polytope $P\ssu\rr^d$
can be counted in polynomial time, for a fixed dimension~$n$ (see also~\cite{B2,BP}).  He utilized the
\emph{short generating function} approach pioneered by Brion, Vergne and others
(see~\cite{B3} for details and references).
Woods~\cite{W1} extended this approach to general Boolean formulas.

In the next breakthrough, Barvinok and Woods showed how to
count projections of integer points in a (single) polytope in polynomial
time~\cite{BW}.  Woods~\cite{W1} extended this approach to general Presburger
expressions~$\Phi$ with a fixed number of inequalities
(see also~\cite{W2} and an alternative proof in~\cite{shortPR}).
As a consequence, he showed that deciding $\SPAd$ is in~$\PP$.
This represents the most general positive result in this direction:

\begin{theo}[Woods]  \label{cor:PR2}
Fix $n_{1},n_{2}$ and $s$.
Given a short Presburger expression $\Phi(\x,\y)$ in variables $\x \in \zz^{n_{1}},\y \in \zz^{n_{2}}$
with at most $s$ inequalities, the sentence
\begin{equation*}
\forall \ts\y \;\; \exists \ts\x \, : \, \Phi(\x,\y)
\end{equation*}
can be decided in polynomial time.   Moreover, the number of solutions
\begin{equation*}
\# \, \big\{ \ts\y \, : \, \exists \ts\x  \;\;  \Phi(\x,\y)\big\}
\end{equation*}
can be computed in polynomial time.
\end{theo}


\subsection{Kannan's Partition Theorem}\label{sec:KPT_intro}
In~\cite{K1}, Kannan introduced
the technology of \emph{test sets} for efficient solutions of~$\PIPP$.
The \emph{Kannan Partition Theorem} (KPT), see Theorem~\ref{th:KPT} below,
claims that one can find in polynomial time
a partition of the $k$-dimensional parameter space $W$ into polynomially
many rational (co-)polyhedra
$$(\circ) \qquad
W \, = \, P_1 \,\sqcup\, P_2  \,\sqcup\, \dots \,\sqcup\, P_r\ts,
$$
so that only a bounded number of tests need to be performed
(see $\S$\ref{ss:kannan-KPT} for precise statement details).

In~\cite{shortPR}, we showed that KPT if valid would imply a polynomial time decision algorithm for $\SPAm$, and in particular $\GIP$ for a restricted system.
Thus, at the time of proving Theorem~\ref{th:NP-Kannan} in~\cite{KannanNPC}, we thought that~\cite{shortPR} and~\cite{KannanNPC} together would completely characterize the complexity of $\GIP$, depending on whether the system is restricted or not.

In view of our theorems~\ref{th:PR3hard},~\ref{cor:system1},~\ref{cor:system2} and~\ref{th:PH}, it strongly suggests that KPT may actually be erroneous.
However, we did not
expect this at the time of writing \cite{shortPR}.  In fact,
the prevailing view was that $\SPAm$ would always be in~$\PP$,
which neatly aligned with the results in~\cite{shortPR} (conditional upon KPT). Now that the hardness results are known, we are actually able combine the current techniques with some of those in~\cite{shortPR} to obtain
the following quantitative result, which strongly contradicts KPT:

\begin{theo}\label{th:KPT-exp}
Fix $m,n$ and let $k=1$.  Let $\phi$ be the total bit length of the matrix
$A \in \zz^{m \times n}$ in $\KPT$.
Then for the number $r$ of pieces in Kannan's partition~$(\circ)$, we must have
$r > \exp(\ep\phi)\ts$ for some constant $\ep=\ep(n,m)>0$.
\end{theo}

We conclude no polynomial size partition~$(\circ)$ exists as claimed by KPT.
See Section~\ref{sec:kannan}  for a detailed presentation of this result and its implications,
$\S$\ref{ss:finrem-KPT} for our point of view, and
$\S$\ref{ss:finrem-KPT-gap} for the gap in the original proof of KPT.

\bigskip

\section{Notations}

\begin{itemize}
\item[]
We use $\nn \ts = \ts \{0,1,2,\ldots\}$ and $\zz_{+} \ts = \ts \{1,2,3,\ldots\}$

\item[]
Universal/existential quantifiers are denoted $\for/\ex$.

\item[]
Unspecified quantifiers are denoted by $Q_{1}, Q_{2}$, etc.


\item[]
Unquantified Presburger expressions are denoted by $\Phi, \Psi$, etc.


\item[]
We use $\left[
\begin{smallmatrix}
a \\
b
\end{smallmatrix}
\right]
$ for a disjunction $(a \lor b)$ and $
\{
\begin{smallmatrix}
a \\
b
\end{smallmatrix}
\}
$ for a conjunction $(a \land b)$.

\item[]
All constant vectors are denoted $\n, \b, \albar, \nubar$, etc.

\item[]
We use $0$ to denote both zero and the zero vector.


\item[]
All matrices are denoted $A, B, C$, etc.

\item[]
All integer variables are denoted $x,y,z$, etc.

\item[]
All vectors of integer variables are denoted $\x, \y, \z$, etc.



\item[]
In a vector $\y = (y_{1},y_{2})$, we draw $y_{2}$ as a vertical and $y_{1}$ as a horizontal coordinate.

\item[]
We use $\floor{.}$ to denote the floor function.

\item[]
The the vector $\y$ with coordinates $y_{i} = \floor{x_{i}}$ is denoted by $\y = \floor{\x}$.




\item[]
Half-open intervals are denoted by $[\al,\be)$, $(\al,\be]$, etc.

\item[]
A \emph{polyhedron} is an intersection of finitely many closed half-spaces in~$\rr^n$.

\item[]
A \emph{copolyhedron} is a polyhedron with possibly some open facets.

\item[]
A \emph{polytope} is a bounded polyhedron.

\item[]
Subsets of $\nn$ are denoted by $\Ga, \De$, etc.


\end{itemize}

\bigskip

\section{Basic properties of finite continued fractions}\label{sec:frac}

\nin
Every rational number $\al > 1$ can be written in the form:
\begin{equation*}
\al \; = \; [a_{0};\; a_{1},\, \dots,\, a_{n}] \; = \; a_{0} + \cfrac{1}{a_{1} + \cfrac{1}{\ddots + \cfrac{1}{a_{n}}}} \; ,
\end{equation*}
where $a_{0},\dots,a_{n} \in \zz_{+}$.
If $a_{n} > 1$, we have another representation:
\begin{equation*}
\al \; = \; [a_{0};\; a_{1},\, \dots,\, a_{n}-1,\, 1] \; = \; a_{0} + \cfrac{1}{a_{1} + \cfrac{1}{\ddots + \cfrac{1}{(a_{n}-1) + \cfrac{1}{1}}}} \; .
\end{equation*}
On the other hand, if $a_{n} = 1$, then we also have:
\begin{equation*}
\al \; = \; [a_{0};\; a_{1},\, \dots,\, a_{n-1},\, 1] \; = \; [a_{0};\; a_{1},\, \dots,\, a_{n-1} + 1].
\end{equation*}

It is well known that any rational $\al > 1$ can be written as a continued fraction
as above in exactly two ways (see e.g.~\cite{Kar,Khin}), one with an odd number of
terms and the other one with an even number of terms.


If a continued fraction $[a_{0};\; a_{1},\, \dots,\, a_{n}]$ evaluates to a rational value $p/q$, we identify it with the integer point $(q,p)$.
We write:
\begin{equation*}
(q,p) \; \lra \; [a_{0};\; a_{1},\, \dots ,\, a_{n}].
\end{equation*}

From now on, we will only consider continued fractions with an odd number of terms:
\begin{equation*}
\al \; = \; [a_{0};\; a_{1},\, \dots,\, a_{2k}].
\end{equation*}
To facilitate later computations, we will relabel these $2k+1$ terms as:
\begin{equation*}
\al \; = \; [a_{0};\; b_{0},\, a_{1},\, b_{1},\, \dots,\, a_{k-1},\, b_{k-1},\, a_{k}].
\end{equation*}
The convergents of $\al$ are $2$-dimensional integer vectors, defined as:
\begin{equation}\label{eq:conv_def}
\aligned
C_{0} = (1,0) & \,,\; D_{0} = (0,1), \\
C_{i} = a_{i-1}D_{i-1} + C_{i-1} & \,,  \text{ for } i = 1, \dots, k+1, \\
D_{i} = b_{i-1}C_{i} + D_{i-1} & \,, \text{ for } i = 1, \dots, k.
\endaligned
\end{equation}
We call $C_{0},D_{0},\dots,C_{k},D_{k}, C_{k+1}$ the convergents for $\al$.
If $C_{i} = (q_{i},p_{i})$ and $D_{i} = (s_{i},r_{i})$ then we have the properties:

\smallskip
\begin{itemize}

\item[P1)] $p_{0}=0$, $q_{0}=1$, $r_{0}=1$, $s_{0}=0$.

\item[P2)] $p_{i}=a_{i-1}r_{i-1} + p_{i-1}$, $q_{i}=a_{i-1}s_{i-1} + q_{i-1}$.

\item[P3)] $r_{i}=b_{i-1}p_{i} + r_{i-1}$, $s_{i}=b_{i-1}q_{i} + s_{i-1}$.

\item[P4)] $C_{i+1} = (q_{i+1},p_{i+1}) \lra [a_{0};\; b_{0},\, a_{1},\, b_{1},\, \dots,\,  b_{i-1},\, a_{i}]$.

\item[P5)] The quotients $p_{i}/q_{i}$ form an increasing sequence, starting with $p_{0}/q_{0} = 0$ and ending with $p_{k+1}/q_{k+1} = \al$.

\item[P6)] $D_{i+1} = (s_{i+1},r_{i+1}) \lra [a_{0};\; b_{0},\, a_{1},\, b_{1},\, \dots,\, a_{i},\, b_{i}]$.

\item[P7)] The quotients $r_{i}/s_{i}$ form a decreasing sequence, starting with $r_{0}/s_{0} = \infty$, and ending with $r_{k}/s_{k} = [a_{0};\; b_{0},\, a_{1},\, b_{1},\, \dots,\, a_{k-1},\, b_{k-1}]$.

\end{itemize}


\begin{figure}[hbt]
\begin{center}
\psfrag{C0}{\small$C_0$}
\psfrag{C1}{\small$C_1$}
\psfrag{C2}{\small$C_2$}
\psfrag{Ck+1}{\small$C_{k+1}$}

\psfrag{D0}{\small$D_0$}
\psfrag{D1}{\small$D_1$}
\psfrag{Dk}{\small$D_k$}

\psfrag{y1}{\small$y_1$}
\psfrag{y2}{\small$y_2$}

\epsfig{file=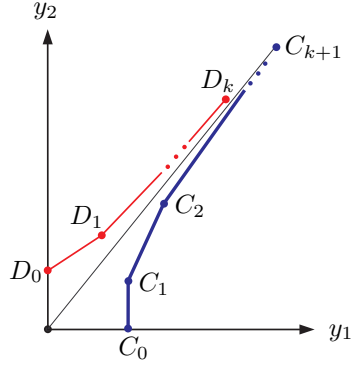,width=5cm}
\end{center}

\vspace{-2em}

\caption{The curves $\mathcal{C}$ (bold) and $\mathcal{D}$.}
\label{f:continued_fraction}
\end{figure}

\smallskip

\nin
Denote by $O$ the origin in $\zz^{2}$. The geometric properties of these convergents are:

\smallskip


\begin{itemize}

\item[G1)] Each vector $\ovv{O C_{i}}$ and $\ovv{O D_{i}}$ is primitive in $\zz^{2}$, meaning $\gcd(p_{i},q_{i}) = \gcd(r_{i},s_{i}) = 1$.

\item[G2)] Each segment $\seg{C_{i} C_{i+1}}$ contains exactly $a_{i}+1$ integer points, since $\ovv{C_{i} C_{i+1}} = a_{i} \ovv{O D_{i}}$.

\item[G3)] Each segment $\seg{D_{i} D_{i+1}}$ contains exactly $b_{i}+1$ integer points, since $\ovv{D_{i} D_{i+1}} = b_{i} \ovv{O C_{i+1}}$.

\item[G4)] The curve $\mathcal{C}$ connecting $C_{0},C_{1},\dots,C_{k+1}$ is (strictly) convex upward (see Figure~\ref{f:continued_fraction}).

\item[G5)] The curve $\mathcal{D}$ connecting $D_{0},D_{1},\dots,D_{k}$ is (strictly) convex downward.

\item[G6)] There are no interior integer points above $\mathcal{C}$ and below $\ovv{O C_{k+1}}$.
In other words, $\mathcal{C}$ is the upper envelope of all non-zero integer points between $\ovv{O C_{0}}$ and $\ovv{O C_{k+1}}$.
\end{itemize}

\bigskip

\section{From arithmetic progressions to short Presburger sentences}\label{sec:main_proof}

\subsection{Covering with arithmetic progressions}\label{sec:convert_covering}

For a triple $(g,h,e)\in\NN^{3}$, denote by $\AP(g,h,e)$ the arithmetic progression:
\begin{equation*}
\AP(g,h,e) \; = \; \{g + je \;:\; 0 \le j \le h\}.
\end{equation*}
We reduce the following classical $\NP$-complete problem to $\SPAt$:

\problemdef{AP-COVER}
{An interval $J = [\mu,\nu] \subset \zz$ and $k$ triples $(g_{i},h_{i},e_{i})$ for $i=1,\dots,k$.}
{Is there $z \in J$ such that $z \notin \AP_{1} \cup \dots \cup \AP_{k}$, where $\AP_{i} = \AP(g_{i},h_{i},e_{i})$?}

The problem
$\textsc{AP-COVER}$ was shown to be $\NP$-complete by Stockmeyer and Meyer (Theorem~\ref{th:SM}).
A short proof of this is included in \S\ref{sec:AP-COVER} for completeness.
We remark that the inputs $\mu,\nu,g_{i},h_{i},e_{i}$ to the problem are in binary.
We can assume that each $h_{i} \ge 1$, i.e., each $\AP_{i}$ contains more than $1$ integer.
This is because we can always increase $\nu \gets \nu+1$ and add the last integer $\nu+1$
to any progression $\AP_{i}$ that previously had only a single element.  Note that
$\textsc{AP-COVER}$ is also invariant under translation, so we can assume that
$\mu,\nu$ and all $g_{i},h_{i},e_{i}$ are positive integers.

Next, let:
\begin{equation*}
M \; = \; 1+ \nu \prod_{i=1}^{k} g_{i}(g_{i} + h_{i} e_{i}).
\end{equation*}
We have:
\begin{equation*}
M > \nu \quad \text{and} \quad M > \max_{i}(g_{i} + h_{i} e_{i}).
\end{equation*}
i.e., the interval $[1,M-1]$ contains $J$ and all $\AP_{i}$.
Moreover, we have:
\begin{equation}\label{eq:gcd_cond}
\gcd(M,g_{i}) = \gcd(M,g_{i}+h_{i}e_{i}) = 1, \; i=1,\dots,k.
\end{equation}
Note that $M$ can be computed in polynomial time from the input of $\textsc{AP-COVER}$, and
\begin{equation*}
\log M \, = \, O \left( \sum_{i=1}^{k} \, \log g_{i} + \log h_{i} + \log e_{i} \right).
\end{equation*}

Let us construct a continued fraction
\begin{equation*}
\al \; = \; [a_{0};\; b_{0},\, a_{1},\, b_{1},\, \dots,\, a_{2k-2},\, b_{2k-2},\, a_{2k-1}]\ts
\end{equation*}
with the following properties:

\smallskip


\begin{itemize}

\item[1)]  All $a_{i},b_{j} \in [1,M]$.

\item[2)] For each $1 \le i < k$, we have $a_{2i} = 1$.

\item[3)] For each $1 \le i \le k$, we have $a_{2i-1} = h_{i}$.

\item[4)] For each $1 \le i \le k$, if $$C_{2i-1} := (q_{2i-1},p_{2i-1}) \lra [a_{0};\; b_{0},\, \dots,\, a_{2i-2}]$$
then we have $p_{2i-1} \equiv g_{i} \mod{M}$.

\item[5)] For each $1 \le i \le k$, if $$C_{2i} := (q_{2i},p_{2i}) \lra [a_{0};\; b_{0},\, \dots,\, a_{2i-1}]$$
then we have
$p_{2i} \equiv g_{i} + h_{i}e_{i} \mod{M}$.

\item[6)] For each $1 \le i \le k$, the segment $\seg{C_{2i-1}C_{2i}}$ contains exactly $h_{i}+1$ integer points. Moreover, the set $$\mathcal{A}_{i} \; := \; \{ y_{2} \; \textup{mod} \; M \; : \; (y_{1},y_{2}) \in \seg{C_{2i-1} C_{2i}}\}$$
is exactly $\AP_{i}$.

\item[7)] For each $1 \le i < k$, the segment $\seg{C_{2i}C_{2i+1}}$ contains no integer points apart from the two end points.

\end{itemize}


\smallskip

We construct $\al$ iteratively as follows.
We say an integer vector $Y = (y_{1},y_{2})$ is congruent to $z$ mod $M$, denoted $Y \equiv z \mod{M}$, if $y_{2} \equiv z \mod{M}$.
As in~\eqref{eq:conv_def}, let $C_{0} = (1,0)$ and $D_{0} = (0,1)$.

\medskip

\begin{enumerate}[leftmargin=7em]
\setlength\itemsep{0.5em}

\item[\bf Step 1:]
Let $a_{0} = g_{1}$. Then
$$C_{1} = a_{0} D_{0} + C_{0} = (1,g_{1}) \;\; \text{and} \;\; C_{1} \equiv g_{1} \mod{M}.$$

\item[\bf Step 2:]
Take $b_{0}$ so that $$D_{1} = b_{0} C_{1} + D_{0}  = (b_{0},b_{0} g_{1}) + (0,1) \equiv e_{1} \mod{M},$$ i.e.,
\begin{equation*}
b_{0}g_{1} + 1 \equiv e_{1} \mod{M}.
\end{equation*}
We can solve for $b_{0}$ mod $M$ because $\gcd(M,g_{1})=1$ from~\eqref{eq:gcd_cond}.
So there exists $b_{0} \in [1,M]$ s.t. $D_{1} \equiv e_{1} \mod{M}$.

\item[\bf Step 3:]
Take $a_{1} = h_{1}$. This implies $$C_{2} = a_{1} D_{1} + C_{1} \equiv h_{1} e_{1} + g_{1} \mod{M}.$$
By Property (G2), we also have exactly $h_{1}+1$ integer points on $\seg{C_{1} C_{2}}$.

\item[\bf Observation:]
After these steps, we have $h_{1}+1$ integer points on $\seg{C_{1} C_{2}}$. Every two such consecutive points differ by $\ovv{O D_{1}}$. Reduced mod $M$, they give:
\begin{equation*}
C_{1} \equiv g_{1},\; g_{1} + e_{1},\; \dots,\; g_{1} + h_{1}e_{1} \equiv C_{2} \mod{M}.
\end{equation*}
Thus, we have $\mathcal{A}_{1} = \AP_{1}$. \ts Conditions (1)--(7) hold so far.

\item[\bf Step 4:]
Take $b_{1}$ so that $D_{2} \equiv g_{2} - (g_{1}+h_{1}e_{1}) \mod{M}$. Since we have the recurrence
\begin{equation*}
D_{2} = b_{1} C_{2} + D_{1}  \equiv b_{1}(g_{1} + h_{1}e_{1}) + e_{1} \mod{M}
\end{equation*}
this is equivalent to solving
\begin{equation*}
b_{1}(g_{1} + h_{1}e_{1}) + e_{1} \equiv g_{2} - (g_{1}+h_{1}e_{1}) \mod{M}.
\end{equation*}
Again we can solve for $b_{1}$ mod $M$ because $\gcd(M,g_{1}+h_{1}e_{1})=1$ from~\eqref{eq:gcd_cond}.
So there exists $b_{1} \in [1,M]$ s.t. $D_{2} \equiv g_{2} - (g_{1}+h_{1}e_{1} ) \mod{M}$.

\item[\bf Step 5:]
Take $a_{2}=1$. This implies
\begin{align*}
\quad \quad C_{3} = a_{2} D_{2} + C_{2} & \equiv  g_{2} - (g_{1}+h_{1}e_{1}) + g_{1} + h_{1}e_{1} \\
& \equiv g_{2} \mod{M}.
\end{align*}
This satisfies condition (4) for $i=2$.
Now we can start encoding $\AP_{2}$ with $C_{3} \mod{M}$.

\item[\bf Observation:]
One can see that $b_{1}$ in Step 4 was appropriately set up to facilitate  Step 5.
It is conceptually easier to start with Step 5 and retrace to get the appropriate condition for $b_{1}$.
Taking $a_{2}=1$ also implies that there are no other integer points on $\seg{C_{2}C_{3}}$ apart from the two endpoints.

\item[\bf Step 6:]
Take $b_{2}$ so that $D_{3} = b_{2} C_{3} + D_{2} \equiv e_{2} \mod{M}$. This is similar to Step 2.
Again we use condition~\eqref{eq:gcd_cond}.

\item[\bf Step 7:]
Take $a_{3} = h_{2}$, which implies $$C_{4} = a_{3}D_{3} + C_{3} \equiv g_{2} + h_{2}e_{2} \mod{M}.$$
After this, we again get exactly $h_{2}+1$ integer points on $\seg{C_{3}C_{4}}$. Reduced mod $M$, they give $\mathcal{A}_{2} = \AP_{2}$.
\ts Note that conditions (1)--(7) still hold.

\item[]
The rest proceeds similarly to Steps 4--7, for $2 \le j \le k-1$:

\item[\bf Step 4$j$:]
Take $b_{2j-1}$ so that $$D_{2j} \equiv g_{j+1} - (g_{j}+h_{j}e_{j}) \mod{M}.$$

\item[\bf Step 4$j$+1:]
Take $a_{2j}=1$, which implies $$C_{2j+1} = D_{2j} + C_{2j} \equiv g_{j+1} \mod{M}.$$

\item[\bf Step 4$j$+2:]
Take $b_{2j}$ so that $D_{2j+1} \equiv e_{j+1} \mod{M}$.

\item[\bf Step 4$j$+3:]
Take $a_{2j+1}=h_{j+1}$, which implies $$C_{2j+2} \equiv g_{j+1} + h_{j+1}e_{j+1} \mod{M}.$$
The segment $\seg{C_{2j+1}C_{2j+2}}$ contains exactly $h_{j+1}+1$ integer points.

\item[\bf Observation:]
After these four steps, we get $\mathcal{A}_{j+1}=\AP_{j+1}$. \ts Conditions (1)--(7) hold throughout.

\end{enumerate}


\medskip

All modular arithmetic mod $M$ in the above procedure can be performed in polynomial time.
The last {\bf Step $4k-1$} gives:
\begin{equation*}
C_{2k} \; = \; (q_{2k},p_{2k}) \; \lra \; [a_{0};\; b_{0},\, a_{1},\, b_{1},\, \dots,\, a_{2k-1}].
\end{equation*}
All terms $a_{i}$ and $b_{j}$ are in the range $[1,M]$, so the final quotient $p_{2k}/q_{2k}$ can be computed in polynomial time using the recurrence~\eqref{eq:conv_def}.
This implies that $p_{2k}$ and $q_{2k}$ have polynomial binary lengths compared to the input $\mu,\nu,g_{i},h_{i},e_{i}$ of $\textsc{AP-COVER}$.
The curve $\mathcal{C}$ connecting $C_{0},C_{1},\dots,C_{2k}$ is shown in Figure~\ref{f:C}.

\begin{figure}[hbt]
\begin{center}
\psfrag{O}{\small$O$}
\psfrag{C0}{\small$C_0$}
\psfrag{C1}{\small$C_1$}
\psfrag{C2}{\small$C_2$}
\psfrag{C3}{\small$C_3$}
\psfrag{C4}{\small$C_4$}
\psfrag{C2k}{\small$C_{2k}$}

\epsfig{file=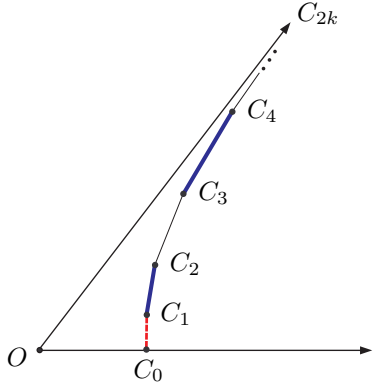, width=5.2cm}
\end{center}

\vspace{-2em}

\caption{The curve $\mathcal{C}$. }

\label{f:C}
\end{figure}

Here each bold segment $\seg{C_{2i-1}C_{2i}}$ contains $h_{i}+1$ integer points. Each thin black segment $\seg{C_{2i}C_{2i+1}}$ contains no interior integer points.
The dotted segment $\seg{C_{0}C_{1}}$ contains $g_{1}+1$ integer points,
the first $g_{1}$ of which we will not need.
Let $\mathcal{C'}$ be $\mathcal{C}$ minus the first $g_{1}$ integer points on $C_{0}C_{1}$.
For brevity, we also denote $C_{2k} = (q_{2k},p_{2k}) = (q,p)$.

\subsection{Analysis of the construction}

We define:
\begin{equation}\label{eq:A_def}
\De \; = \; \big\{\ts z \;\, : \;\,  \ex (y_{1},y_{2}) \in \mathcal{C'} \quad z \equiv y_{2} \mod{M} \ts\big\}.
\end{equation}
By condition (7), every integer point  $\y = (y_{1},y_{2}) \in \mathcal{C'}$ lies on one of the segments $\seg{C_{1}C_{2}}$, $\seg{C_{3}C_{4}}$, $\dots$, $\seg{C_{2k-1} C_{2k}}$.
Moreover, by condition (6), for $1 \le i \le k$ we have:
\begin{equation*}
\aligned
\AP_{i} \; = \; \mathcal{A}_{i} \; = \; \big\{ z \, : \, \ex \y \in \seg{C_{2i-1}C_{2i}} \;\; z \equiv y_{2} \mod{M} \ts\big\}
\endaligned
\end{equation*}
Therefore, we have:
\begin{equation*}
\AP_{1} \cup \dots \cup \AP_{k} \; = \; \mathcal{A}_{1} \cup \dots \cup \mathcal{A}_{k} \; = \; \De.
\end{equation*}
Recall that $\textsc{AP-COVER}$ asks whether:
\begin{equation*}
\ex z \in J  \;\; z \notin \AP_{1} \cup \dots \cup \AP_{k} \;\; \iff \;\; \ex z \in J \;\; z \notin \De.
\end{equation*}
By~\eqref{eq:A_def}, this is equivalent to:
\begin{equation*}
\ex z \in J  \quad \for \y \in \mathcal{C'} \quad  z  \not\equiv y_{2} \mod{M},
\end{equation*}
which can be rewritten as:
\begin{equation}\label{eq:intermediate}
\ex z \in J  \;\; \for \y \;\;  z  \not\equiv y_{2} \mod{M}  \; \lor \; \y \notin \mathcal{C'}.
\end{equation}

Next, we express the condition $\y = (y_{1},y_{2}) \in \mathcal{C'}$ in short Presburger arithmetic.
Let $\v = (p,-q)$ and $\theta$ be the cone between $\ovv{OC_{0}}$ and $\ovv{OC_{2k}}$, i.e.,
\begin{equation*}
\theta \; = \; \big\{\ts \y \in \rr^{2} \; : \; y_{2} \ge 0 \ts,\; \v\cdot\y \ge 0 \ts\big\}.
\end{equation*}
For each $\y = (y_{1},y_{2}) \in \theta$, denote by $P_{\y}$ the parallelogram with two opposite vertices $O$ and $\y$ and sides parallel to $\ovv{OC_{0}}$ and $\ovv{OC_{2k}}$ (see Figure~\ref{f:P}).
We also require that horizontal edges in $P_{\y}$ are open, i.e.,
\begin{equation}\label{eq:P_y_def}
P_{\y} = \left\{
\x \in \R^{2}  \, : \,
\begin{matrix}
\v \cdot \y \; \ge \; \v \cdot \x \; \ge \; 0\\
y_{2} \; > \; x_{2} \; > \; 0
\end{matrix}
\right\}.
\end{equation}

\vspace{-1em}
\begin{figure}[hbt]
\begin{center}
\psfrag{O}{\small$O$}
\psfrag{y}{\small$\y$}
\psfrag{Py}{\small$P_{\y}$}
\psfrag{p}{\small$p$}
\psfrag{q}{\small$q$}
\psfrag{y1}{\small$y_1$}
\psfrag{y2}{\small$y_2$}
\psfrag{C2k}{\small$C_{2k}$}

\epsfig{file=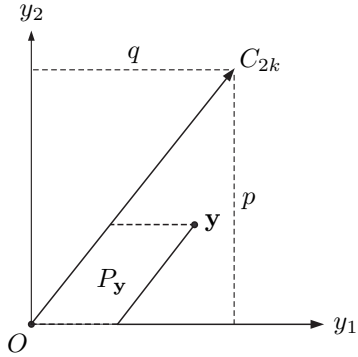, width=5cm}
\end{center}

\vspace{-1.5em}

\caption{The parallelogram $P_{\y}$. The upper and lower edges of $P_{\y}$ are open (dotted). Here we denote $C_{2k} = (q_{2k},p_{2k}) = (q,p)$.}
\label{f:P}
\end{figure}

\begin{lem}\label{lem:easy}
For $\y \in \zz^2$, we have:
\begin{equation}\label{eq:C_condition}
\y \in \mathcal{C'} \; \iff \; \v \cdot \y \ge 0  \; \land \; y_{2} \ge g_{1}  \; \land \;  P_{\y} \cap \zz^{2} = \varnothing
.
\end{equation}
\end{lem}


\begin{proof}
First, assume $\y := (y_{1},y_{2}) \in \mathcal{C'}$.
Recall that $\mathcal{C'}$ is $\mathcal{C}$ minus the first $g_{1}$ integer points on $\seg{C_{0}C_{1}}$.
Therefore, we have $y_{2} \ge g_{1}$.
Since $\mathcal{C}$ sits inside $\theta$, we also have $\y \in \theta$, which implies $\v \cdot \y \ge 0$.
Let $\mathcal{R}$ be the concave region above $\mathcal{C}$ and below $\ovv{OC_{2k}}$.
By property (G6), $\mathcal{R}$ contains no interior integer points.
Since $\y \in \mathcal{C}$, we have $P_{\y} \subset \mathcal{R}$. Therefore, the parallelogram $P_{\y}$ in~\eqref{eq:P_y_def} contains no integer points.
We conclude that $\y$ satisfies the RHS in~\eqref{eq:C_condition}.

Conversely, assume $\y$ satisfies the RHS in~\eqref{eq:C_condition} but $\y \notin \mathcal{C'}$.
The following argument is illustrated in Figure~\ref{f:P_proof}.
First, $\v \cdot \y \ge 0  \, \land \, y_{2} \ge g_{1}$ implies $\y \in \theta$.
Also, the parallelogram $P_{\y}$ contains no integer points.
By property (G6), if $\y \notin \mathcal{C'}$, it must lie strictly below $\mathcal{C'}$.
Let $\x$ and $\x'$ be the integer points on $\mathcal{C}$ that are immediately above and below $\y$ (see Figure~\ref{f:P_proof}).
In other words, $\x \in \mathcal{C}$ is the integer point immediately above the intersection of $\mathcal{C}$ with the upper edge of $P_{\y}$, and $\x' \in \mathcal{C}$ is the integer point immediately below the intersection of $\mathcal{C}$ with the right edge of $P_{\y}$.
Since $P_{\y}$ contains no integer points, particularly those on $\mathcal{C}$, the points $\x$ and $\x'$ must be adjacent on $\mathcal{C}$, i.e., they form a segment on $\mathcal{C}$.\footnote{Note that $\x$ and $\x'$ are not necessarily two consecutive vertices $C_i$ and $C_{i+1}$ of $\mathcal{C}$. They could be two consecutive points on some segment $\seg{C_{i}C_{i+1}}$.}
Now we draw a parallelogram $D$ with two opposite vertices $\x,\x'$ and edges parallel to those of $P_{\y}$ (the dashed bold parallelogram in Figure~\ref{f:P_proof}).
It is clear that $D$ lies inside $\theta$ and also contains $\y$.
Take $\y'$ to be the reflection of $\y$ across the midpoint of $\x\x'$.
Since $\x, \x'$ and $\y$ are integer points, so is $\y'$.
We also have $\y' \in D \subset \theta$.
Note also that $\y'$ lies on the opposite side of $\mathcal{C}$  compared to $\y$.
Therefore, we have $\y' \in \mathcal{R}$, contradicting property $(G6)$.
\end{proof}

\begin{figure}[hbt]
\begin{center}
\psfrag{C}{\small $\mathcal{C}$}
\psfrag{y}{{\small $\y$}}
\psfrag{y'}{{\small $\y'$}}
\psfrag{x}{{\small $\x$}}
\psfrag{x'}{{\small $\x'$}}
\psfrag{Py}{\small $P_{\y}$}
\psfrag{O}{\small $O$}
\psfrag{D}{{\small $D$}}
\epsfig{file=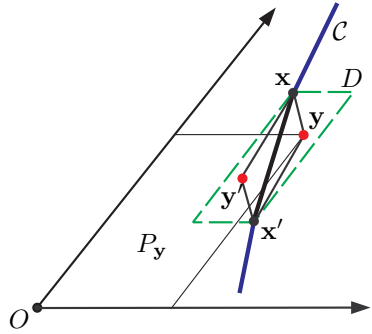, width=5.4cm}
\end{center}

\vspace{-1.5em}

\caption{$\y'$ is the reflection of $\y$ across the midpoint of $\x\x'$.}
\label{f:P_proof}
\end{figure}

\begin{rem}\label{rem:exception}
There is a subtle point about the existence of $\x'$ in the above proof.
It is clear that $\x$ exists because $\y$ lies below $\mathcal{C}$.
However, if $\y$ lies too low, the right edge $P_{\y}$ might not intersect $\mathcal{C}$.
For example, in Figure~\ref{f:exception}, we have $g_{1}=1$ and $\y$ lies on the line $y_{2}=1$.
This this case, $P_{\y}$ contains no integer points and its right edge does not intersect $\mathcal{C}$.
Thus, we have no $\x'$ and the geometric argument in Figure~\ref{f:P_proof} does not work.
However, this can be easily fixed by requiring $a_{0} = g_{1} \ge 2$, noting that $\textsc{AP-COVER}$ is invariant under a simultaneous translation of $J$ and all $\AP_{i}$.
\end{rem}

\vspace{-1em}
\begin{figure}[hbt]
\begin{center}
\psfrag{C}{$\mathcal{C}$}
\psfrag{C0}{$C_{0}$}
\psfrag{C1}{$C_{1}$}
\psfrag{C2k}{$C_{2k}$}
\psfrag{y}{$\y$}
\psfrag{Py}{$P_{\y}$}
\psfrag{O}{$O$}
\psfrag{y1}{$y_{1}$}
\psfrag{y2}{$y_{2}$}

\epsfig{file=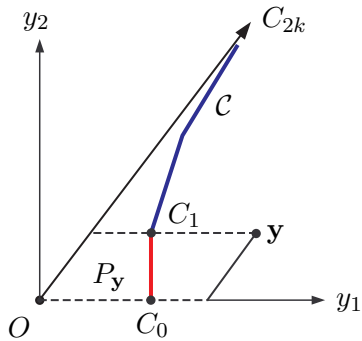, width=5.2cm}
\end{center}

\vspace{-1.5em}

\caption{Here $g_{1}=1$, $\y \notin \mathcal{C}$, and yet $P_{\y}$ contains no integer points (dotted edges are open).}
\label{f:exception}
\end{figure}

\subsection{Proof of Theorem~\ref{th:PR3hard} (decision part)}

Combining~\eqref{eq:intermediate},~\eqref{eq:P_y_def} and~\eqref{eq:C_condition}, the negation of $\textsc{AP-COVER}$ is equivalent to:
\begin{equation}\label{eq:semifinal}
\ex z \in J  \quad \for \y \quad  \Bigg[ z \not\equiv y_{2}   \mod{M}   \; \lor \;  \v \cdot \y < 0 \; \lor \; y_{2} < g_{1} \; \lor \; \ex \x
\left\{
\begin{matrix}
\v \cdot \y  \ge  \v \cdot \x  \ge  0\\
y_{2} \; > \; x_{2} \; > \; 0
\end{matrix}
\right\}
\Bigg]
.
\end{equation}
The condition $z \not\equiv y_{2} \mod{M}$ can be expressed as:
\begin{equation*}
\ex t \quad 0 < z-y_{2}-Mt < M.
\end{equation*}
This existential quantifier $\ex t$ can be absorbed into $\ex \x$ because they are connected by a disjunction.
The restricted quantifier $\ex z \in J$ with $J = [\mu,\nu]$ is just
\begin{equation*}
\ex z  \quad  \mu \le z \le \nu. 
\end{equation*}
Overall, we can rewrite~\eqref{eq:semifinal} in prenex normal form:
\begin{equation}\label{eq:final}
\aligned
\ex z  \;\; \for \y \;\; \ex \x \;\; \mu \le z \le \nu \; \land \; \Bigg[ \;&   0 < z-y_{2}-Mx_{1} < M  \; \lor \\
& \lor \; \v \cdot \y < 0 \; \lor  \; y_{2} < g_{1} \; \lor \;
\left\{
\begin{matrix}
\v \cdot \y \ge \v \cdot \x \ge 0\\
y_{2} \;>\; x_{2} \;>\; 0
\end{matrix}
\right\}
\Bigg].
\endaligned
\end{equation}

All strict inequalities with integer variables can be sharpened.
For example $y_{2} > x_{2}$ is equivalent to $y_{2} -1 \ge x_{2}$.
This final form contains $5$ variables and $10$ inequalities.

In summary, we have reduced (the negation of) $\textsc{AP-COVER}$ to~\eqref{eq:final}. This shows that~\eqref{eq:final} is $\NP$-hard, and so is $\SPAt$.
For $\NP$-completeness, by Theorem 3.8 in~\cite{G}, if $\SPAt$ is true, there must be a satisfying $\z$ with binary length bounded polynomially in the binary length of $\Phi$.
Given such a polynomial length certificate~$\z$, one can substitute it into $\SPAt$ and verify the rest of the sentence, which has the form $\for \y \; \ex \x \; \Psi(\x,\y)$.
Here $\Psi$ is again a short Presburger expression.
By Corollary~\ref{cor:PR2}, this can be checked in polynomial time.
Thus, the whole sentence $\SPAt$ is in $\NP$.
This concludes the proof of the decision part of Theorem~\ref{th:PR3hard}.
\hfill$\sq$

\bigskip

\section{Proof of theorems~\ref{cor:system1} and~\ref{cor:system2} (decision part)}\label{sec:system_proofs}

We will recast~\eqref{eq:final} into the form $\GIP$.
For the polytopes $R$ and $Q$ in $\GIP$, let $R = J = [\mu,\nu]$ and
\begin{equation}\label{eq:Q_def}
Q = \big\{\ts \y \in \rr^{2} \; : \; y_{2} \ge g_{1},\; y_{1} \le q,\; \v \cdot \y \ge 0 \ts\big\}\ts,
\end{equation}
see Figure~\ref{f:Q}.

\begin{figure}[hbt]
\begin{center}
\psfrag{O}{\small$O$}
\psfrag{y1}{\small$y_1$}
\psfrag{y2}{\small$y_2$}
\psfrag{C}{\small$\mathcal{C'}$}
\psfrag{C0}{\small$C_0$}
\psfrag{C1}{\small$C_1$}
\psfrag{C2k}{\small$C_{2k}=(q,p)$}
\psfrag{g1}{\small$g_1$}
\psfrag{Q}{\hspace{-1em} \small $Q$}
\epsfig{file=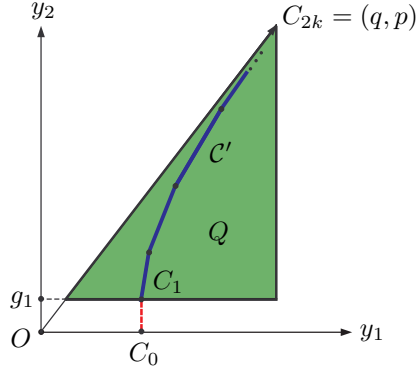, width=5.2cm}
\end{center}
\vspace{-1em}
\caption{The triangle $Q$ (shaded). }
\label{f:Q}
\end{figure}

Since $R \supset \mathcal{C'}$,~\eqref{eq:intermediate} is equivalent to:
\begin{equation*}
\ex z \in R \quad \for \y \in Q \quad z \not\equiv y_{2} \mod{M} \;\; \lor \;\; \y \notin \mathcal{C'}.
\end{equation*}
By condition~\eqref{eq:C_condition}, for $\y \in Q$, we have
\begin{equation*}
\y \notin \mathcal{C'} \quad \iff \quad \ex\x \in P_{\y}\ts.
\end{equation*}
Thus, the sentence~\eqref{eq:final} is equivalent to:
\begin{equation}\label{eq:restricted}
\aligned
\ex z \in R \quad & \for \y \in Q \quad \ex \x  \quad
\\
& 0 <\ts z-y_{2}-Mx_{1} <\ts M \quad \lor \quad \x \in P_{\y}\ts.
\endaligned
\end{equation}

The remaining step is to covert the expression
\begin{equation}\label{eq:disjunction}
1 \le z-y_{2}-Mx_{1} \le M-1 \; \lor \;
\left\{
\begin{matrix}
\v \cdot \y \ge \v \cdot \x \ge 0\\
y_{2}-1 \ge x_{2} \ge 1
\end{matrix}
\right\}
\end{equation}
into a single system.
Here we expanded $\x \in P_{\y}$ and also sharpened all inequalities.

First, observe that for $z \in R$ and $\y \in Q$, there exists $\x$ satisfying~\eqref{eq:disjunction} if and only if there exists such an $\x$ within some bounded range.
Indeed, both $R$ and $Q$ are bounded, and~\eqref{eq:disjunction} imply boundedness for $\x$.
Therefore, we can take an $N$ large enough so that
\begin{equation}\label{eq:bound}
-N \le z,\, y_{1},\, y_{2},\, x_{1},\, x_{2} \le N.
\end{equation}
For instance, $N = (M + p + q)^{3}$ suffices.

Now we convert~\eqref{eq:disjunction} into a single system.
This can be done in two slightly different ways, leading to theorems~\ref{cor:system1} and~\ref{cor:system2}.

\subsection{Proof of Theorem~\ref{cor:system1}  (decision part)}
Applying the distributive law on~\eqref{eq:disjunction}, we get an equivalent expression:
\begin{equation}\label{eq:distribute}
\left[
\begin{matrix}
1 \ts\le\ts z-y_{2} - Mx_{1} \ts\le\ts M-1 \\
\v \cdot \x \ts \le \ts \v \cdot \y
\end{matrix}
\right] \quad \land \quad
\left[
\begin{matrix}
1 \ts\le\ts z-y_{2}-Mx_{1} \ts\le\ts M-1 \\
0 \ts\le\ts \v \cdot \x
\end{matrix}
\right]
\quad \land \quad \dots
\end{equation}
Here each $\left[
\begin{smallmatrix}
a \\
b
\end{smallmatrix}
\right]
$
stands for a disjunction $a \lor b$ of two terms. In total, there are four such disjunctions.

Now we convert each of the above disjunctions into a conjunction.
WLOG, consider the first one in~\eqref{eq:distribute}.
By the bounds~\eqref{eq:bound}, it is equivalent to:
\begin{equation}\label{eq:disjunct}
\left[
\begin{matrix}
1 \ts\le\ts z-y_{2}-Mx_{1} \ts\le\ts M-1 \\
0 \ts\le\ts \v \cdot \y - \v \cdot \x \ts\le\ts 2N(p+q)
\end{matrix}
\right].
\end{equation}
Let $t_{1}=z-y_{2}-Mx_{1}$ and $t_{2}=\v \cdot \y - \v \cdot \x$.
By~\eqref{eq:bound}, we always have
\begin{equation*}
|t_{1}| \le 2N + MN, \quad |t_{2}| \le 2N(p+q).
\end{equation*}
Define two polygons in $\rr^{2}$:
$$P_{1} \;=\; \big\{\ts (t_{1},t_{2}) \in \rr^{2} \;:\;
1 \le t_{1} \le M-1, |t_{2}| \le 2N(p+q)
 \ts\big\},$$
$$P_{2} \;=\; \big\{\ts (t_{1},t_{2}) \in \rr^{2} \;:\;
|t_{1}| \le 2N + MN, \ 0 \le t_{2} \le 2N(p+1)
 \ts\big\}.$$
Then~\eqref{eq:disjunct} can be rewritten as:
\begin{equation}\label{eq:disjunct_restated}
(t_{1},t_{2}) \; \in \; P_{1} \; \cup \; P_{2} \,.
\end{equation}

Next, define:
\begin{equation*}
P_{1}' = (P_{1},0),\quad P_{2}'=(P_{2},1) \quad \text{and} \quad P=\text{conv}(P'_{1},P'_{2}).
\end{equation*}
In other words, we embed $P_{1}$ into the plane $t_{3}=0$ and $P_{2}$ into the plane $t_{3}=1$, all inside $\rr^{3}$.
As $3$-dimensional polytopes, the convex hull of $P'_{1}$ and $P'_{2}$ is another polytope $P \subset \rr^{3}$.
It is easy to see that $P$ has $6$ facets, whose equations can be found from the vertices of $P_{1}$ and $P_{2}$.
Also observe that for $(t_{1},t_{2},t_{3}) \in \zz^{3}$, we have:
\begin{equation*}
(t_{1},t_{2},t_{3}) \; \in \; P  \; \iff \ \;
\aligned
& (t_{1},t_{2}) \; \in \; P_{1}\ts, \ t_{3}=0\ts, \ \, \text{or} \\
& (t_{1},t_{2}) \; \in \; P_{2}\ts, \ t_{3}=1\ts.
\endaligned
\end{equation*}
From this, we have:
\begin{equation}\label{eq:union}
(t_{1},t_{2}) \; \in \; P_{1}  \; \cup \; P_{2}  \; \iff \; \ex t_{3}  \; : \; (t_{1},t_{2},t_{3}) \; \in \; P .
\end{equation}
Combined with~\eqref{eq:disjunct_restated}, it implies that~\eqref{eq:disjunct} is equivalent to:
\begin{equation*}
\ex t \; : \; (z-y_{2}-Mx_{1},\, py_{1}-qy_{2}-px_{1}+qx_{2},\, t) \; \in \; P.
\end{equation*}
The above condition is a linear system with $6$ equations.
Doing this for each disjunction in~\eqref{eq:distribute}, we get four new variables $\t \in \zz^{4}$ and a combined system of $24$ inequalities.
Thus, the original disjunction~\eqref{eq:disjunction} is equivalent to a system:
\begin{equation*}
\ex \t \in \zz^{4} \; : \;
 A\ts\bx \. + \. B\ts\by \. + \. Cz \. + D\ts\t    \, \le \, \ov b\ts.
\end{equation*}
The inner existential quantifiers $\ex\x \in \zz^{2}$ and $\ex\t \in \zz^{4}$ can be combined into $\ex \x \in \zz^{6}$.
Substituting everything into~\eqref{eq:restricted}, we obtain the decision part of Theorem~\ref{cor:system1}.\hfill$\sq$

\subsection{Proof of Theorem~\ref{cor:system2} (decision part)}
Another way to convert~\eqref{eq:disjunction} into a system is to directly interpret its two clauses and two separate polytopes.
The same bounds~\eqref{eq:bound} still apply.
We will need the following special case of the \emph{Upper Bound Theorem}
(see e.g.\ Theorem~8.23 and Exercise~0.9 in~\cite{Z}).

\begin{theo}[McMullen]\label{lem:upper_bound}
A polytope $P \subset \rr^{d}$ with $n$ vertices has at most
\begin{equation*}
\quad
f(d,n) \; := \;
\left(
\begin{matrix}
n - \lceil d/2 \rceil \\
n-d
\end{matrix}
\right)
+
\left(
\begin{matrix}
n-\floor{d/2}-1 \\
n-d
\end{matrix}
\right) \quad \text{facets.}
\end{equation*}
Similarly, a polytope $Q \subset \rr^{d}$ with $n$ facets has at most $f(d,n)$ vertices.
\end{theo}

\smallskip

The first polytope we consider is given by:
\begin{equation*}
\big\{(x_{1},y_{2},z) \in \rr^{3} \; : \; 1 \. \le \. z-y_{2}-Mx_{1} \. \le \. M-1,
\; -N \le x_{1},y_{2},z \le N \ts \big\}.
\end{equation*}
This is a $3$-dimensional polytope with $8$ facets.
Applying Theorem~\ref{lem:upper_bound}, we see that it has at most $12$ vertices.
To interpret it as a polytope in $z, \y$ and $\x$ we need to form its direct
product with the interval $-N \le y_{2} \le N$ also embed it in the hyperplane $x_{2}=0$.
This produces a polytope $P_{1} \subset \rr^{5}$ with $24$ vertices.

The second polytope we consider is given by:
\begin{equation*}
\big\{
(\x,\y) \in \rr^{4} \; : \; \v \cdot \y \;\ge\; \v \cdot \x \;\ge\; 0,\;
y_{2}-1 \;\ge\; x_{2} \;\ge\; 1,\; y \in Q
\big\}.
\end{equation*}
As a $4$-dimensional polytope it has only $8$ vertices.
These $8$ vertices correspond to the cases when $\y$ lies at one of the three vertices of~$Q$.
Two of these vertices give two degenerate parallelograms $P_{\y}$, each of which is a segment with $2$ vertices.
The lower right vertex of $Q$ gives a non-degenerate parallelogram $P_{\y}$ with $4$ vertices.
To interpret this as a $5$-dimensional polytope in $z,\y$ and $\x$, we need to form its direct product with the polytope $R = [\mu,\nu]$ for $z$.
This results in a polytope $P_{2} \subset \rr^{5}$ with $16$ vertices.

Altogether, we have two polytopes $P_{1},P_{2} \subset \rr^{5}$ with $40$ vertices in total.
We reapply the ``lifting'' trick in~\eqref{eq:union} to produce another polytope $P \subset \rr^{6}$ with $40$ vertices so that:
\begin{equation*}
(z,\y,\x) \; \in \; P_{1}  \; \cup \; P_{2}  \; \iff \;  \ex t \; : \; (z,\y,\x,t) \; \in \; P.
\end{equation*}
By Theorem~\ref{lem:upper_bound}, the resulting polytope $P$ has at most
$$f(6,40) \, = \, \binom{37}{34} \. + \. \binom{36}{34} \, = \, 8400
$$
facets, which can all be found in polynomial time from the vertices.
Therefore, the disjunction~\eqref{eq:disjunction} is equivalent to a system:
\begin{equation*}
\ex t \; : \;  A\ts\bx \. + \. B\ts\by \. + \. Cz \. + D\ts t    \, \le \, \ov b\ts
\end{equation*}
with at most $8400$ inequalities.
The existential quantifiers $\ex t$ and $\ex \x \in \zz^{2}$ can be combined into $\ex \x \in \zz^{3}$.
Substituting all into~\eqref{eq:restricted}, we obtain the decision part of
Theorem~\ref{cor:system2}. \hfill$\sq$

\bigskip

\section{Proof of theorems~\ref{th:PR3hard}, \ref{cor:system1} and~\ref{cor:system2} (counting part)}

\nin
Notice that the above reduction from $\textsc{AP-COVER}$ to~\eqref{eq:final} is parsimonious,
i.e., $z$ lies in $J \cpl (\AP_{1} \cup \dots \cup \AP_{k})$ if and only if $\mu \le z \le \nu$ and
\begin{equation}\label{eq:count_final}
\for \y \;\; \ex \x \;\;  \Bigg[ 0 < z-y_{2}-Mx_{1} < M   \;  \lor \; \v \cdot \y < 0  \; \lor \; y_{2} < g_{1} \;  \lor
\left\{
\begin{matrix}
\v \cdot \y \ge \v \cdot \x \ge 0\\
y_{2} > x_{2} > 0
\end{matrix}
\right\}
\Bigg].
\end{equation}

At the same time, the reduction from $\textsc{3SAT}$ to $\textsc{AP-COVER}$ given in \S\ref{sec:AP-COVER} is also parsimonious, i.e., every satisfying assignment $\u$ for~\eqref{eq:3SAT} corresponds to a unique $z \in J$ not covered by the arithmetic progressions and vice versa.
This is due the uniqueness part of the Chinese Remainder Theorem used in~\eqref{eq:moduli}.
Since $\textsc{\#3SAT}$ is $\sharpP$-complete (see e.g.~\cite{AB,MM,Pap}), so is counting the number of $z$ satisfying~\eqref{eq:count_final}.
This proves the second part of Theorem~\ref{th:PR3hard}.

The counting parts of theorems~\ref{cor:system1} and~\ref{cor:system2} can be proved with a similar argument to Section~\ref{sec:system_proofs}.\hfill$\sq$



\bigskip

\section{Proof of Theorem~\ref{th:PH}}


Consider the following $m$-generalization of the problem $\textsc{AP-COVER}$:

\problemdef
{$m$-AP-COVER}
{
The following elements:

$\bu$ $m$ intervals $J_{1} = [\mu_{1},\nu_{1}]$ $,\dots,$ $J_{m} = [\mu_{m},\nu_{m}]$,

$\bu$ $k_{1}$ triples $(g_{1i},\ts h_{1i},\ts e_{1i})$, with $1 \le i \le k_{1}$,

$\quad\ldots$

$\bu$ $k_{m}$ triples $(g_{mi},\ts h_{mi},\ts e_{mi})$, with $1 \le i \le k_{m}$,

$\bu$ $m$ integers $\tau_{1},\, \dots,\, \tau_{m} \in \zz$.
}
{
\begin{equation*}
\hspace{-3.2em}
\aligned
Q_{1} (z_{1} \in  J_{1}  \cpl \Delta_{1}) \;\; & \dots \;\;  Q_{m-1} (z_{m-1} \in J_{m-1} \cpl \Delta_{m-1}) \\
& \dots \;\; Q_{m} (z_{m} \in J_{m}) \;\; : \;\; \tau_{1}z_{1} + \ldots + \tau_{m}z_{m} \notin \Delta_{m}.
\endaligned
\end{equation*}
Here $Q_{1},\dots,Q_{m} \in \{\for,\ex\}$ are $m$ alternating quantifiers with $Q_{m} = \ex$. The sets $\De_{1},\dots,\De_{m}$ are defined as:
\begin{equation*}
\De_{t} = \AP_{t1} \cup \dots \cup \AP_{tk_{t}}, \; 1 \le t \le m
\end{equation*}
where
\begin{equation*}
\AP_{ti} = \AP(g_{ti},h_{ti},e_{ti}), \; 1 \le i \le k_{t}.
\end{equation*}
}


Using Theorem~\ref{th:m-AP-COVER}, we prove Theorem~\ref{th:PH} by reducing $\textsc{$m$-AP-COVER}$ to short Presburger arithmetic.
Theorem~\ref{th:PR3hard} is the special case when $m=1$ ($\SigmaP_{1} \equiv \NP$).
For simplicity, we show the reduction for the case $m=2$.
The same argument works for $m>2$.

Consider $\textsc{$2$-AP-COVER}$ in~\eqref{eq:2-AP-COVER}, which is $\PiP_{2}$-complete.
We can rewrite it as:
\begin{equation}\label{eq:2-AP-COVER-restated}
\for z_{2} \in J_{2} \quad \bigl[ \ts  z_{2}  \in \De_{2}  \;\; \lor \;\; \ex z_{1} \in J_{1} \;\; \tau_{1}z_{1} + \tau_{2}z_{2} \notin \De_{1}  \ts \bigr].
\end{equation}
Replacing $z$ with $\ts\tau_{1}z_{1}+\tau_{2}z_{2}\ts$ in~\eqref{eq:count_final}, we can express the condition $\tau_{1}z_{1} + \tau_{2}z_{2} \notin \De_{1}$ by a short formula  $\for \y \; \ex \x \; \Phi_{1}(\x,\y,\tau_{1}z_{1}+\tau_{2}z_{2})$ with $4$ extra variables $\x,\y \in \zz^{2}$ and $8$ linear inequalities.
Similarly, the condition $z_{2} \in \De_{2}$ can be expressed as $\ex \w \; \for \t \; \Phi_{2}(\t,\w,z_{2})$ with another $4$ variables $\w,\t \in \zz^{2}$ and also $8$ inequalities.

Overall,~\eqref{eq:2-AP-COVER-restated} is equivalent to:
\begin{equation*}
\for z_{2} \in J_{2} \quad \biggl[ \ts  \ex \w \;\; \for \t \;\; \Phi_{2}(\t,\w,z_{2})  \;\; \lor \;\; \ex z_{1} \in J_{1} \;\; \for \y \;\; \ex \x \;\; \Phi_{1}(\x,\y,\tau_{1}z_{1}+\tau_{2}z_{2})  \ts \biggr]\,.
\end{equation*}
Each of the restricted quantifiers $\for z_{2} \in J_{2}$ and $\ex z_{1} \in J_{1}$ contributes $2$ more inequalities.
Note that the two quantifier groups $\ex \w \; \for \t$ and $\ex z_{1} \; \for \y \; \ex \x$ can be merged through the disjunction into $\ex \w \; \for \y' \; \ex \x$.
This results in new variables $\w \in \zz^{2}$, $\y' = (\t,\y) \in \zz^{4}$ and $\x \in \zz^{2}$.
The final sentence takes the form
\begin{equation*}
\for z_{2} \quad \ex \w \quad \for \y' \quad \ex \x \quad \Phi(\x,\y',\w,z_{2})
\end{equation*}
with $20$ inequalities and $9$ variables ($z_{1}$ has been absorbed into $\w$).\hfill$\sq$

\bigskip

\section{Bilevel optimization and Pareto optima}\label{sec:optimization}





\subsection{Proof of Theorem~\ref{th:minmax}}\label{ss:minmax}
First, we characterize the convex chains $\mathcal{C}$ and $\mathcal{D}$ from Figure~\ref{f:continued_fraction} using a quadratic function:

\begin{lem}\label{lem:quad_cond}
	Let $\al = p/q \in \Q_{+}$. If $\u,\v \in \zz^{2}$ satisfy $\frac{u_{2}}{u_{1}} < \al < \frac{v_{2}}{v_{1}}$ and $v_{2}u_{1} - v_{1}u_{2} = 1$ then both $\frac{u_{2}}{u_{1}}$ and $\frac{v_{2}}{v_{1}}$ are ``weak'' convergents of $\al$, i.e., $\u \in \mathcal{C}$ and $\v \in \mathcal{D}$.
\end{lem}

\begin{proof}
	Assume $\u \notin \mathcal{C}$,
	then $\u = (u_{1},u_{2})$ lies stricly below $\mathcal{C}$.
	By the argument from Lemma~\ref{lem:easy}, the parallelogram $P_{\u}$ contains another point $\u' = (u'_{1},u'_{2}) \in \zz^{2}$ with $\frac{u'_{2}}{u'_{1}} < \al$.
	Draw a line $\ell$ parallel to $\vec{\v}$ and passing through $\u$.
	Since $\frac{v_{2}}{v_{1}} > \al$, $P_{\u}$ lies completely to the left of $\ell$ (See Figure~\ref{f:u_v}).
	From this, we conclude that $1 = v_{2}u_{1} - v_{1}u_{2} > v_{2} u'_{1} - v_{1} u'_{2} > 0$.
	In other words, the triangle $O\u\v$ has larger area than that of  $O\u'\v$.
	This is impossible, because $v_{2} u'_{1} - v_{1} u'_{2} \in \zz$.
	Therefore, we must have $\u \in \mathcal{C}$.
	By the same argument, we have $\v \in \mathcal{D}$.
\end{proof}

\vspace{-1em}
\begin{figure}[hbt]
	\begin{center}
		\psfrag{C}{$\mathcal{C}$}
		\psfrag{x}{{\small $y_{1}$}}
		\psfrag{y}{{\small $y_{2}$}}
		\psfrag{O}{{\small $O$}}
		\psfrag{y=ax}{{\small $y_{2} = \al y_{1}$}}
		\psfrag{u}{{\small $\u$}}
		\psfrag{v}{{\small $\v$}}
		\psfrag{u'}{{\small $\u'$}}
		\psfrag{Pu}{{\small $P_{\u}$}}
		\psfrag{l}{{\small $\ell$}}
		\epsfig{file=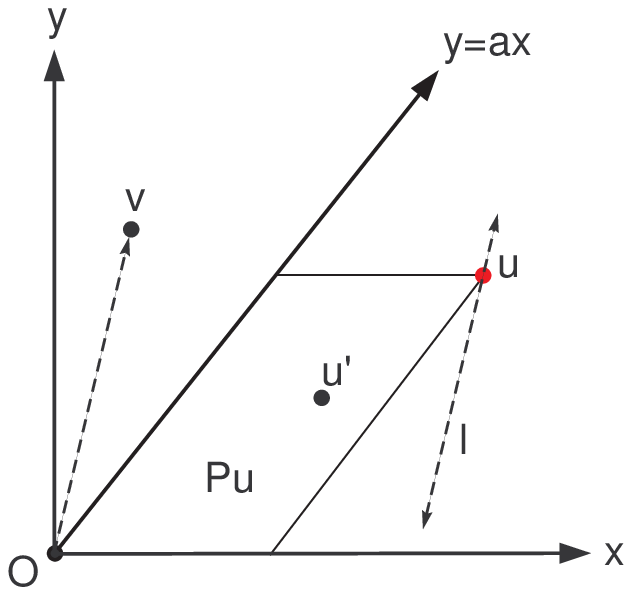, width=5.8cm}
	\end{center}
	
	\vspace{-2em}
	
	\caption{$\u$ and $\v$.}
	\label{f:u_v}
\end{figure}

Conversely, for any weak convergent $\u \in \mathcal{C}$, we can find $\v \in \mathcal{D}$ with $v_{2}u_{1} - v_{1} u_{2} = 1$.
This comes from the fact that any two consecutive convegents $\frac{p_{i}}{q_{i}}$ and $\frac{p_{i+1}}{q_{i+1}}$ of $\alpha$ satisfy $p_{i+1}q_{i} - p_{i}q_{i+1} = (-1)^{i}$.

\begin{proof}[Proof of Theorem~\ref{th:minmax}]
	We use the same reduction from \textsc{AP-COVER} as in Sections~\ref{sec:main_proof} and~\ref{sec:system_proofs}.
	With the same rational number $\al = p/q$, let
	\begin{equation*}
	Q = \big\{\ts (u_{1},u_{2}) \in \rr^{2} \; : \; u_{2} \ge g_{1},\; u_{1} \le q,\; pu_{1}-qu_{2} \ge 0 \ts\big\}\ts,
	\end{equation*}
	and
	\begin{equation*}
	P = \big\{\ts (v_{1},v_{2}) \in \rr^{2} \; : \; v_{2} \le p-1,\; v_{1} \ge 0,\; pv_{1}-qv_{2} \le 0 \ts\big\}\ts.
	\end{equation*}
	
	\begin{figure}[hbt]
		\begin{center}
			\psfrag{O}{\small $O$}
			\psfrag{y1}{\small $y_1$}
			\psfrag{y2}{\small $y_2$}
			\psfrag{C'}{\small $\mathcal{C'}$}
			\psfrag{(p,q)}{\small $(p,q)$}
			\psfrag{P}{\small $P$}
			\psfrag{Q}{\small $Q$}
			\epsfig{file=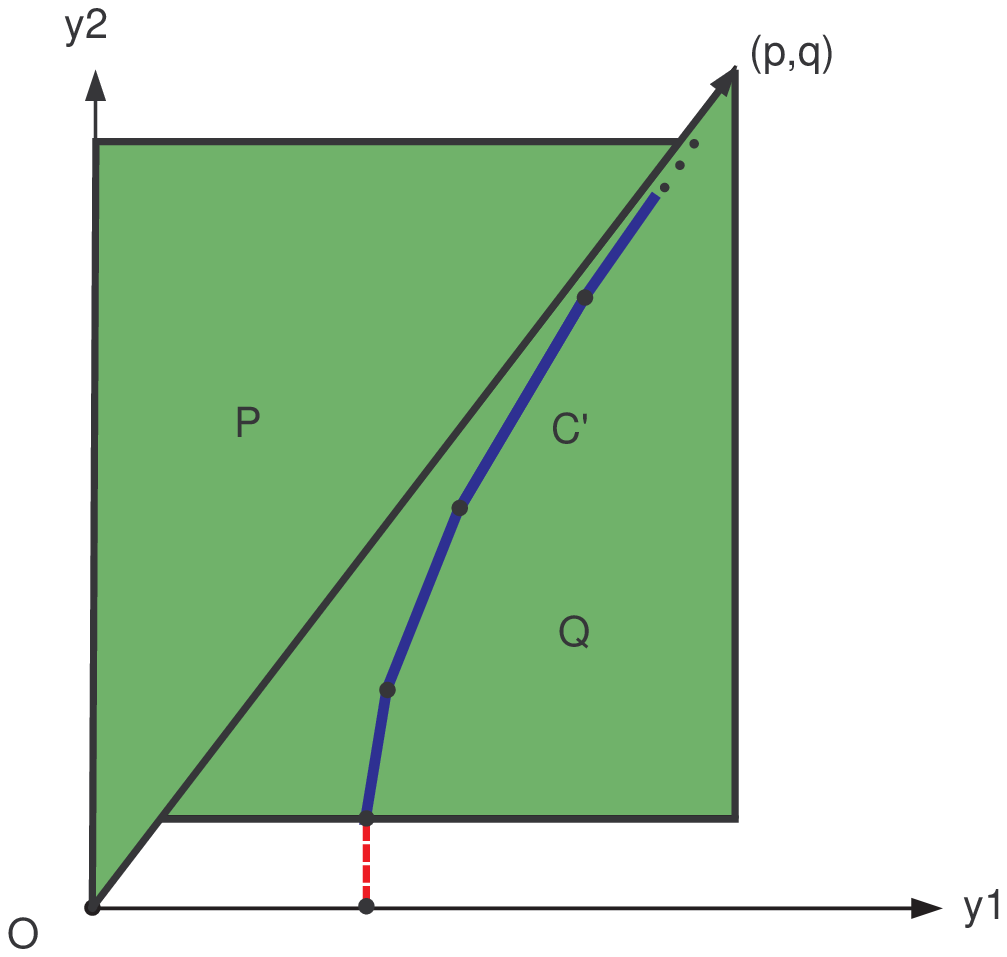, width=5.0cm}

		\end{center}
		
		\vspace{-2em}
		
		\caption{$P$ and $Q$.}
	\end{figure}

	Recall from~\eqref{eq:intermediate} that the $\NP$-complete problem \textsc{AP-COVER} asks if there exists some $z \in J \subset [0,M]$ for which no $\y \in \mathcal{C'}$ satisfies $z \equiv y_{2} \mod M$.
	Here $\mathcal{C'}$ is the part of the convex chain $\mathcal{C}$ lying inside $Q$.
	Now let $\w = (\u,\v,t)$, $W = Q \times P \times [0,T]$  and
	\begin{equation*}
	h(z,\w) \quad = \quad K(v_{2}u_{1} - v_{1}u_{2} - 1) \quad + \quad (u_{2} - z - tM)^{2}.
	\end{equation*}
	Here $T$ and $K$ are two appropriately chosen constants.
	Specifically, let $T = p/M$ so that if $z \equiv u_{2} \mod M$ then there always exists $t \in [0,T]$ with $t = \frac{u_{2}-z}{M}$.
	For $K$, we pick it sufficiently large so that $K \gg (u_{2}-z-tM)^{2}$ for every $\u \in Q$, $z \in J$ and $t \in [0,T]$.
	Clearly $K = (2TM + p)^{3}$ suffices.

	With $\u \in Q \cap \zz^{2}$ and $\v \in P \cap \zz^{2}$, we have $v_{2}u_{1} - v_{1}u_{2} \ge 1$.
	Furthermore, by Lemma~\ref{lem:quad_cond}, equality happens if and only if $\u \in \mathcal{C'}$ and $\v \in \mathcal{D}$.
	For a fixed $z \in J$ consider the $\w \in W$ that  minimizes $h(z,\w)$.
	Since $K \gg (z - tM - u_{2})^{2}$, the first term in $h$ always dominate the second one.
	So we must have $v_{2}u_{1}-v_{1}u_{2} = 1$ when $h$ is minimized, which implies $\u \in \mathcal{C'}$.
	Furthermore, among all $\y \in \mathcal{C'}$, $\u$ must be the one for which $u_{2} \text{ mod } M$ is closest to $z$, so that the second term in $h$ is minimized.
	Thus,
	\begin{equation*}
	\min_{\w \in W \cap \zz^{5}} \quad h(z,\w)  \quad \ge \quad  0,
	\end{equation*}
	and equality holds if and only if there is some $\y \in \mathcal{C'}$ with $z \equiv y_{2} \mod M$.
	Therefore,
	\begin{equation*}
	\max_{z \in J \cap \zz} \quad \min_{\w \in W \cap \zz^{5}} \quad h(z,\w)  \quad > \quad  0
	\end{equation*}
	if and only if there exists some $z \in J$ for which no $\y \in \mathcal{C'}$ satisfies $z \equiv y_{2} \mod M$.
	We conclude that computing~\eqref{eq:minmax} is $\NP$-hard, as it implies \textsc{AP-COVER}.
\end{proof}

\subsection{Proof of Theorem~\ref{th:Pareto}}\label{ss:Pareto}

First recall the definition of Pareto optima defined in Section~\ref{ss:application}.
To summarize Section~\ref{ss:minmax}, we showed that computing
\begin{equation}\label{eq:minmax_restated}
	\max_{z \in J \cap \zz} \quad \min_{\w \in W \cap \zz^{5}} \quad h(z,\w)
\end{equation}
is $\NP$-hard for $I \subset \rr^{1}$ an interval, $W \subset \rr^{5}$ a polytope with $18$ facets and $h : \rr^{6} \to \rr$ a quadratic function.
Let $Q = I \times W \subset \rr^{6}$, which has $38$ facets.
For $\x = (z,\w) \in Q \cap \zz^{6}$, let
\begin{equation*}
f_{1}(\x) = z, \quad f_{2}(\x) = -z \quad \text{and} \quad f_{3}(\x) = h(z,\w).
\end{equation*}

Consider the set of Pareto minima of $(f_{1},f_{2},f_{3})$ on~$Q$.
For convenience, we denote an outcome vector $\y = \big(f_{1}(\x), f_{2}(\x), f_{3}(\x)\big)$ by $\y = f(\x)$.
Consider two points $\x=(z,\w)$ and $\x'=(z,\w')$ in $Q \cap \zz^{6}$. If $h(z,\w) < h(z,\w')$ then $f_{1}(\x) = f_{1}(\x')$, $f_{2}(\x) = f_{2}(\x')$, and $f_{3}(\x) < f_{3}(\x')$. Then $\y' = f(\x')$ is not a Pareto minimum in this case.
Therefore, all Pareto minima must be of the form $\y = f(\x)$, where $\x = (z,\w_{\min})$ with $h(z,\w_{\min}) = \min_{\w \in W \cap \zz^{5}} h(z,\w)$.
Furthermore, if $\x = (z, \w_{\min})$ and $\x' = (z', \w'_{\min})$ are two such points with $z \ne z'$, then the outcome vectors $\y = f(\x)$ and $\y' = f(\x')$ are incomparable, simply because either $f_{1}(\x) < f_{1}(\x')$ and $f_{2}(\x) > f_{2}(\x')$, or the other way around.

We conclude that the set Pareto minima of $(f_{1},f_{2},f_{3})$ on $Q$ is given as:
\begin{equation*}
\mathcal{P} \, = \, \Big\{\ts\y = \big(z,\ts -z,\ts h(z,\w_{\min})\big) \; : \; z \in J \cap \zz,\; h(z,\w_{\min}) = \min_{\w \in W \cap \zz^{5}} h(z,\w) \ts\Big\}.
\end{equation*}
For $\y \in \rr^{3}$, let $g(\y) = -y_{3}$. Then minimizing $g(\y)$ over $\y \in \mathcal{P}$ is the same as
computing the negated value of~\eqref{eq:minmax_restated}.  This proves the first part of Theorem~\ref{th:Pareto}.

To show the hardness of approximating $\min_{\y \in \mathcal{P}}{g(\y)}$ within a multiplicative factor of~$1/2$, recall from Section~\ref{ss:minmax} that the value of~\eqref{eq:minmax_restated} determines the~\textsc{AP-COVER}.
To be precise,~\eqref{eq:minmax_restated} is equal to the largest squared distance of an integer $z \in J$ from the union $\AP_{1} \cup \dots \cup \AP_{k}$, which is $0$ if and only if $J \cap \Z$ is entirely covered by these $\AP$s.

Recall the part of the proof of Theorem~\ref{th:SM}, where we reduce $\textsc{3SAT}$ to $\textsc{AP-COVER}$.
There, we pick the first $\ell$ primes $p_{1}=2,\ts p_{2},\ts \dots,\ts p_{\ell}$.
The reduction would work verbatim if we picked $p_{2}=3,\ts \dots,\ts p_{\ell+1}$ instead.
The advantage of this small change is that now we can exclude the arithmetic progression $z \equiv 0 \mod 2$
from~$J$. In other words, we require $z \equiv 1 \mod 2$  and the Chinese Remainder Theorem still works.
Then the final union $\AP_{1} \cup \dots \cup \AP_{k}$ which we exclude from $J$ must contain all even numbers.
This implies that the largest squared distance of an integer $z \in J$ to $\AP_{1} \cup \dots \cup \AP_{k}$
is at most~$1$.
Therefore, the value of~\eqref{eq:minmax_restated} is either $1$ or~$0$. So getting a $1/2$-approximation
is equivalent to deciding \textsc{AP-COVER}, and thus $\NP$-hard.

\bigskip

\section{Covering with arithmetic progressions}\label{sec:AP}

\subsection{$\NP$-completeness of $\textsc{AP-COVER}$}\label{sec:AP-COVER}
Recall the following problem from $\S$\ref{sec:convert_covering}.

\problemdef{AP-COVER}
{An interval $J = [\mu,\nu] \subset \zz$ and $k$ triples $(g_{i},h_{i},e_{i})$ for $i=1,\dots,k$.}
{Is there $z \in I$ such that $z \notin (\AP_{1} \cup \dots \cup \AP_{k})$, where $\AP_{i} = \AP(g_{i},h_{i},e_{i})$?}

In this section, we reproduce (in a somewhat different language) the original proof
from~\cite{MS}, see also Remark~\ref{r:AP-GJ} below.  The reduction in the proof
will later be extended to work with more quantifiers.

\begin{theo}[Stockmeyer and Meyer]  \label{th:SM}
	$\textsc{AP-COVER}$ is $\NP$-complete.
\end{theo}

\begin{proof}
	We reduce $\textsc{3SAT}$ to $\textsc{AP-COVER}$.
	Consider a $3$-CNF Boolean expression:
	\begin{equation}\label{eq:3SAT}
	\Psi(\u) \; = \; \bigwedge_{i=1}^{n} C_{i}(\u),
	\end{equation}
	where $\u = u_{1} \dots \ts u_{\ell} \in \{\text{true},\text{false}\}^{\ell}$
	are Boolean variables, and each clause $C_{i}(\u)$ is a disjunction of three literals from the set
	\begin{equation*}
	\{u_{j},\, \lnot u_{j} \; : \; 1 \le j \le \ell\}.
	\end{equation*}
	
	Let $p_{1},\dots,p_{\ell}$ be the first $\ell$ primes. We have $p_{\ell} = O(\ell \log \ell)$ by the Prime Number Theorem.
	So $p_{1},\dots,p_{\ell}$ can be found in time $\polyin{\ell}$.
	We restrict $z$ to the interval $J = [0,\ts p)$, where $p = p_{1} \cdots p_{\ell}$.
	For each assignment of $\u = u_{1} \dots \ts u_{\ell} \in \{\text{true},\text{false}\}^{\ell}$, we shall associate a unique integer $z \in J$ that satisfies:
	\begin{equation}\label{eq:moduli}
	u_{j} = \text{true} \; \iff \; z \equiv 1 \mod{p_{j}} \quad ; \quad
	u_{j} = \text{false} \; \iff \; z \equiv 0 \mod{p_{j}}\ts.
	\end{equation}

	First, for each $j$, we exclude all moduli mod $p_{j}$ that are not $0$ or $1$.
	In other words, we exclude the arithmetic progressions:
	\begin{equation}\label{eq:only_01}
	\AP_{jt} \; = \; \big\{ z \in J \; : \; z \equiv t \mod{p_{j}} \big\}
	\quad \text{for} \quad 1 \le j \le \ell,\; 2 \le t < p_{j}\ts.
	\end{equation}
	If $z \notin \bigcup_{jt} \AP_{jt}$ then $z$ is equal to $0$ or $1$ mod every $p_{j}$.
	Now consider each clause $C_{i}(\u)$.
	For example, assume $C_{1}(\u) = u_{1} \lor \lnot u_{2} \lor u_{3}$.
	The negation $\lnot C_{1}(\u)$ is $\lnot u_{1} \land  u_{2} \land \lnot u_{3}$.
	To this, we associate an arithmetic progression:
	\begin{equation}\label{eq:congruent}
	\AP_{1} \; = \; \big\{ z \in J \; : \; z \equiv 0 \mod{p_{1}} \; \land \; z \equiv 1 \mod{p_{2}} \; \land \; z \equiv 0 \mod{p_{3}} \big\}.
	\end{equation}
	By the Chinese remainder theorem, we can write:
	\begin{equation*}
	\AP_{1} \; = \; \big\{ z \in J \; : \; z \equiv e \mod{p_{1}p_{2}p_{3}} \big\},
	\end{equation*}
	where $e$ is unique mod $p_{1}p_{2}p_{3}$ and also computable in polynomial time.
	Then we have:
	\begin{equation}\label{eq:each_clause}
	C_{1}(\u) = \text{true} \quad \iff \quad z \notin \AP_{1}.
	\end{equation}
	Doing this for all clauses $C_{1},\dots,C_{n}$, we get $n$ arithmetic progressions $\AP_{1},\dots,\AP_{n}$.
	From~\eqref{eq:3SAT}, \eqref{eq:only_01} and~\eqref{eq:each_clause}, we conclude that:
	\begin{equation*}
	\Psi(\u) \; = \; \bigwedge_{i=1}^{n} C_{i}(\u) = \text{true} \quad \iff \quad z \; \notin \; \bigcup_{\substack{1 \le i \le n}} \AP_{i} \; \bigcup_{\substack{1 \le j \le \ell \\ 2 \le t < p_{j}}} \AP_{jt} \, .
	\end{equation*}
	Therefore,
	\begin{equation*}
	\ex \u \quad \Psi(\u) = \text{true} \quad \iff \quad  \ex z \in J \;\; : \;\; z \;
	\notin \;  \bigcup_{\substack{1 \le i \le n}} \AP_{i} \; \bigcup_{\substack{1 \le j
			\le \ell \\ 2 \le t < p_{j}}} \AP_{jt}  \,.
	\end{equation*}
	The above LHS is a $\textsc{3SAT}$ sentence, which is $\NP$-complete to decide.
	Thus, the RHS, which is $\textsc{AP-COVER}$, is also $\NP$-complete.
	In total, we have $k \coloneqq n + \sum_{j=1}^{\ell} (p_{j}-1)$ arithmetic progressions,
	each of which can be given as a triple $(g_{i},h_{i},e_{i})$.
\end{proof}

\begin{rem}\label{r:AP-GJ}
	In~\cite[$\S$A7]{GJ}, the problem $\textsc{AP-COVER}$ is phrased differently
	under the name $\textsc{SIMULTANEOUS INCONGRUENCES}$ problem.
\end{rem}

\subsection{Generalization of \textsc{AP-COVER} to $m$ quantifiers}\label{sec:2-AP-COVER}


We consider the following $m$-generalization of the problem $\textsc{AP-COVER}$.

\problemdef
{$m$-AP-COVER}
{
	The following elements:
	
	$\bu$ $m$ intervals $J_{1} = [\mu_{1},\nu_{1}]$ $,\dots,$ $J_{m} = [\mu_{m},\nu_{m}]$,
	
	$\bu$ $k_{1}$ triples $(g_{1i},\ts h_{1i},\ts e_{1i})$, with $1 \le i \le k_{1}$,
	
	$\quad\ldots$
	
	$\bu$ $k_{m}$ triples $(g_{mi},\ts h_{mi},\ts e_{mi})$, with $1 \le i \le k_{m}$,
	
	$\bu$ $m$ integers $\tau_{1},\, \dots,\, \tau_{m} \in \zz$.
}
{
	The truth of the sentence:
	\begin{equation*}
	\hspace{-3.2em}
	\aligned
	Q_{1} (z_{1} \in  J_{1}  \cpl \Delta_{1}) \;\; & \dots \;\;  Q_{m-1} (z_{m-1} \in J_{m-1} \cpl \Delta_{m-1}) \\
	& \dots \;\; Q_{m} (z_{m} \in J_{m}) \;\; : \;\; \tau_{1}z_{1} + \ldots + \tau_{m}z_{m} \notin \Delta_{m}.
	\endaligned
	\end{equation*}
	Here $Q_{1},\dots,Q_{m} \in \{\for,\ex\}$ are $m$ alternating quantifiers with $Q_{m} = \ex$. The sets $\De_{1},\dots,\De_{m}$ are defined as:
	\begin{equation*}
	\De_{t} = \AP_{t1} \cup \dots \cup \AP_{tk_{t}}, \; 1 \le t \le m
	\end{equation*}
	where
	\begin{equation*}
	\AP_{ti} = \AP(g_{ti},h_{ti},e_{ti}), \; 1 \le i \le k_{t}.
	\end{equation*}
}

\vspace{-1em}

For example, $\textsc{$2$-AP-COVER}$ asks whether
\begin{equation}\label{eq:2-AP-COVER}
\for (z_{2} \in J_{2} \cpl \De_{2}) \quad \ex z_{1} \in J_{1} \quad \tau_{1}z_{1} + \tau_{2}z_{2} \notin \De_{1},
\end{equation}
i.e., for all $z_{2} \in J_{2}$
either $z_{2}$ is covered by some $\AP$ in the first group, or there is some
$z_{1} \in J_{1}$ so that their linear combination $\tau_{1}z_{1} + \tau_{2}z_{2}$
is not covered by any $\AP$ in the second group.

\begin{theo}\label{th:m-AP-COVER}
	$\textsc{$m$-AP-COVER}$ is $\SigmaP_{m}$-complete for $m$ odd and $\PiP_{m}$-complete for $m$ even.
\end{theo}



\begin{proof}
	For simplicity, we show that $\textsc{$2$-AP-COVER}$ is $\PiP_{2}$-complete. The proof for general $\textsc{$m$-AP-COVER}$ is analogous.

	This is similar to Theorem~\ref{th:SM}'s proof, but instead of $\textsc{3SAT}$ we decide:
	\begin{equation}\label{eq:2-3SAT}
	\for \v \quad \ex \u \quad \Psi(\u,\v) = \text{true},
	\end{equation}
	where $\u,\v \in \{\text{true},\text{false}\}^{\ell}$, and
	$\Psi(\u,\v) \,=\, \bigwedge_{i=1}^{n} C_{i}(\u,\v)$,
	with each clause $C_{i}(\u,\v)$ a disjunction of three literals from the set
	\begin{equation*}
	\{u_{j},\, \lnot u_{j},\, v_{j},\, \lnot v_{j} \; : \; 1 \le j \le \ell\}.
	\end{equation*}
	
	Deciding~\eqref{eq:2-3SAT} is $\PiP_{2}$-complete (see e.g.~\cite{GJ,Pap}).
	To reduce~\eqref{eq:2-3SAT} to~\eqref{eq:2-AP-COVER}, we again take the first
	$2\ell$ primes $p_{1},\dots,p_{\ell},\, q_{1},\dots,q_{\ell}$.
	Let $\.p=p_{1} \cdots p_{\ell}\.$, $\.q=q_{1} \cdots q_{\ell}\.$ and:
	\begin{equation*}
	J_{1} \; \coloneqq \; [0,\; p) \quad \text{and} \quad J_{2} \; \coloneqq \; [0,\; q).
	\end{equation*}
	Since $\gcd(p,q)=1$, we can also find in polynomial time $\tau_{1},\tau_{2} \in \zz$ so that:
	\begin{equation}\label{eq:tau_def}
	\tau_{1} \equiv 1 \mod{p},\; q \,|\, \tau_{1} \quad \text{and} \quad \tau_{2} \equiv 1 \mod{q},\; p \,|\, \tau_{2}\ts.
	\end{equation}

	Next, we require that $z_{2} \equiv 0 \; \text{or} \; 1 \mod{q_{j}}$ for $i=1,\dots,\ell$.
	This can be expressed as $z_{2} \in J_{2} \cpl \De_{2}$, where $\De_{2}$ is a union of some arithmetic progressions similar to those in~\eqref{eq:only_01}.
	These are the $k_{2}$ progressions $\AP_{21},\dots,\AP_{2k_{2}}$.
	
	We also require $z_{1} \equiv 0 \; \text{or} \; 1 \mod{p_{j}}$ for $j=1,\dots,\ell$.
	By~\eqref{eq:tau_def}, this is equivalent to $\tau_{1}z_{1} + \tau_{2}z_{2} \equiv 0 \; \text{or} \; 1 \mod{p_{j}}$.
	Again, this condition can be expressed as:
	\begin{equation}\label{eq:Gamma}
	\tau_{1}z_{1} + \tau_{2}z_{2} \; \notin \; \Ga_{1}
	\end{equation}
	for $\Ga_{1}$ a union of some arithmetic progressions.
	
	Analogous to~\eqref{eq:moduli}, the variables $z_{1}$ and $z_{2}$ correspond to $\u$ and $\v$, respectively.
	By the Chinese remainder theorem (see~\eqref{eq:congruent} and~\eqref{eq:each_clause}), we can express each clause $C_{i}(\u,\v)$ as:
	\begin{equation*}
	C_{1}(\u,\v) = \text{true} \quad \iff \quad \tau_{1}z_{1} + \tau_{2}z_{2} \; \notin \; \AP_{i}
	\end{equation*}
	for some arithmetic progression $\AP_{i}$ with $i = 1,\dots,n$.
	Let $\De_{1}$ be the union of $\Ga_{1}$ in~\eqref{eq:Gamma} with $\AP_{1},\dots,\AP_{n}$\ts.
	
	Overall, we have $k_{1}+k_{2}$ finite arithmetic progressions from $\De_{1}$ and $\De_{2}$.
	Note that $k_{1}+k_{2}$ is still polynomial compared to $\ell$ and the length of $\Psi$.
	It is straightforward that~\eqref{eq:2-AP-COVER} and~\eqref{eq:2-3SAT}  are equivalent.
	Therefore, deciding~\eqref{eq:2-AP-COVER} is $\PiP_{2}$-complete.
\end{proof}

\bigskip

\section{On Kannan's Partition Theorem} \label{sec:kannan}

\subsection{Validity of KPT} \label{ss:kannan-KPT}
By \emph{Parametric Integer Programming} (PIP), we mean the following problem.
Given an integer matrix $A \in \zz^{m \times n}$ and a $k$-dimensional polyhedron $W \subset \rr^{m}$, is the following sentence true:
\begin{equation}\label{eq:PIP}
\for\ts\b \in W \quad \ex \x \in \zz^{n} \;\; : \;\; A\ts\x \, \le \b.
\end{equation}
We think of $\b$ as a parameter varying over~$W$.
For every fixed $\b$, this gives an Integer Programming problem in fixed dimension $n$.
In~\cite[Theorem~3.1]{K1}, Kannan claimed the following result,
which implies a polynomial time algorithm to decide~\eqref{eq:PIP}.
From here on, we use RA to denote \emph{rational affine transformations}.
Also let $K_{\b} := \{\x \in \R^{n} : A\x \le \b\}$ for every $\b \in W$.

\begin{theo}[Kannan's Partition Theorem]\label{th:KPT}
Fix $n$ and $k$. Given a $\PIP$ problem,
we can find in polynomial time a partition
\begin{equation}\label{eq:KPT_partition}
W = P_1 \sqcup P_2  \sqcup \dots \sqcup P_r,
\end{equation}
where each $P_{i}$ is a rational copolyhedron\footnote{A copolyhedron is a
convex polyhedron with possibly some open facets.}, so that the partition
satisfies the following properties. For each $P_i$, we can find in polynomial
time a finite set $\T_i = \{(S_{ij},T_{ij})\}$ of pairs of RAs
$\ts T_{ij} : \R^{m} \to \R^{n}$ and $S_{ij} : \Z^{n} \to \Z^{n}$,
so that for every $\b \in P_i$ we have:
\begin{equation*}
K_{\b} \cap \Z^{n} \neq \varnothing \; \iff \; \ex (S_{ij},T_{ij}) \in \T_{i} \; : \; S_{ij} \floor{ T_{ij} \b } \in K_{\b}.
\end{equation*}
Furthermore, for each $P_i$, the set $\T_{i}$ contains at most $n^{4n}$ pairs $(S_{ij},T_{ij})$.
The number of all $P_{i}$ is $r \le (m n \phi)^{k n^{\delta n}}$, where $\phi$ is the binary length of $A$ and $\delta$ is a universal constant.
\end{theo}

KPT claims that in order to solve for an $\x \in \Z^{n}$ satisfying $A\x \le \b$ with $\b$ varying over~$W$, we only need to preprocess the matrix $A$ in polynomial time and obtain a polynomial number of regions $P_{i}$.
When queried with $\b \in P_i$, we only need to check for a fixed number ($n^{4n}$) of candidates of the form $\x = S_{ij} \floor{ T_{ij} \b }$ to get an integer solution in $K_{\b}$ (if any exists).




Let us prove that $\KPT$, if true, would imply far stronger statements for a
PIP problems that involves only a matrix of fixed length~$m$. From now on, fix $m,n$ and $k$.
By KPT and the observation $mn \le \phi$, the number of regions $P_{i}$ in~\eqref{eq:KPT_partition}
 can be bounded as:
\begin{equation}\label{eq:num_pieces}
r \; \le \; (m n \phi)^{k n^{\delta n}} \; \le \; \phi^{\ga(n,k)} \..
\end{equation}
Here $\ga(n,k)$ is a constant which depends only on $n$ and $k$.
The following structural result is an implication of KPT when the parameter space $W$ is $1$-dimensional, i.e. when $k=1$~:
\begin{equation}\label{eq:1_dim_param}
W \; = \; \{f(y) \in \rr^{m} \; : \; y \in I\}
\end{equation}
where $f : \rr^{1} \to \rr^{m}$ is a RA, and $I \subset \rr$ a bounded interval.

\begin{lem}\label{lem:one_segment}
Assume~\eqref{eq:num_pieces} holds.
Given a $\PIP$ problem
with a $1$-dimensional parameter space $W$ \eqref{eq:1_dim_param},
there exists a finite set $\T = \{(S_{j},T_{j})\}$  of pairs of RAs $\ts T_{j} : \R^{1} \to \R^{n}$ and $S_{j} : \Z^{n} \to \Z^{n}$ so that the following hold.
For every $y \in I \cap \zz$ and $\ts\b = f(y) \in \rr^{m}$, we have:
\begin{equation*}
K_{\b} \cap \Z^{n} \neq \varnothing \quad \iff \quad \ex (S_{j},T_{j}) \in \T \; : \; S_{j} \floor{ T_{j} y } \in K_{\b}.
\end{equation*}
Furthermore, the set $\T$ contains at most $c(n)$ pairs $(S_{j},T_{j})$,
where $c(n)$ is a constant which depends only on~$n$.
\end{lem}

\begin{rem}
The above lemma says that the bound~\eqref{eq:num_pieces} as implied by $\KPT$ would guarantee a small set of candidates for any ``short'' PIP problem $A\x \le f(y)$ with $1$-dimensional parameters $y$.
The number of candidates $c(n)$ depends only on the dimension $n$.
\end{rem}

\begin{proof}[Proof of Lemma~\ref{lem:one_segment}]
WLOG, assume $I = [0,N)$ and $A = (a_{ij}) \in \zz^{m \times n}$. Let
\begin{equation}\label{eq:M}
M \; = \; N \; \prod_{ij} (|a_{ij}|+1) \;  \prod_{k} (|p_{k}q_{k}| + 1),
\end{equation}
where $p_{k}/q_{k}$ runs over all rational coefficients in $f$.
Let $J = [0,MN)$.
Consider the following $\PIP$ problem with one parameter $y' \in J$ and $n+2$ integer variables $\x\in\zz^{n}$, $y_{1},y_{2} \in \zz$:
\begin{equation}\label{eq:add_var}
Ny_{1} + y_{2} = y',\quad 0 \le y_{1} < M,\quad 0 \le y_{2} < N,\quad A\x - f(y_{2}) \le 0.
\end{equation}
Observe that when~\eqref{eq:add_var} is feasible, the values of $y_{1}$ and $y_{2}$ are uniquely determined.
Indeed, we should have $y_{1} = \floor{y'/N}$ and $y_{2}=y' - Ny_{1}$.
So as $y'$ varies over $J \cap \zz$, the solutions of~\eqref{eq:add_var} correspond bijectively with the solutions of the original PIP problem $A\x \le f(y)$ where $y = \floor{y'/N} \in I$.

Clearly,~\eqref{eq:add_var} can be put into the form $B\z \le g(y')$ where $\z=(\x,y_{1},y_{2}) \in \zz^{n+2}$ are variables and $g$ is an RA.
Let $\b' = g(y')$, then the problem takes the form $B\z \le \b'$.
Also let $W' = \{\b' = g(y') : y' \in J\}$.
Applying KPT to the PIP problem $B\z \le \b'$ with a $1$-dimensional parameter space $W'$, we have a partition of $W'$ into polynomially many intervals.
Since $\b'=g(y')$ and $g$ is an RA, this partition induces another partition on $J$ (the space for $y'$) into intervals:
\begin{equation}\label{eq:int_partition}
J \; = \; J_{1} \sqcup \dots \sqcup J_{r}\ts.
\end{equation}

By~\eqref{eq:num_pieces}, the number $r$ of all intervals in this partition is polynomial in the binary length of the matrix $B$.
From~\eqref{eq:M} and~\eqref{eq:add_var}, it is clear that $B$ has no more than $2mn$ entries, each bounded by $M$.
Therefore, we have:
\begin{equation}\label{eq:r_M}
r \; \le \; \Bigg( \sum_{ij} \; \lceil \log b_{ij} \rceil \Bigg)^{\ga} \; \le \; (2mn \log M)^{\ga} \; \ll \; M.
\end{equation}
Here $\ga=\ga(n,k)$ is some constant degree guaranteed by KPT.
Since $r \ll M$, some interval $J_{i}$ from~\eqref{eq:int_partition} must contain an entire subinterval $I'=[kN,(k+1)N)$ for some $0 \le k < M$.
For simplicity, assume $I'=[kN,(k+1)N] \subseteq J_{1}$.

Also by KPT, for the interval $J_{1}$,  there is a set of candidates $\T_{1} = \{(S_{1j},T_{1j})\}$ of size at most $c(n) \coloneqq (n+2)^{4(n+2)}$ for the PIP problem $B\z \le \b'$.
For every $y' \in I' \subseteq J_{1}$, each solution of~\eqref{eq:add_var} should have $y_{1}=k$ and $y_{2} = y' - Nk$.
By a translation $y = y' - Nk$, we can map $I'$ back to $I$.
Accordingly, we can modify each candidate $(S_{ij},T_{ij}) \in \T_{i}$ to be a pair of RAs in $y$.
Clearly, they serve as candidates for the original PIP problem $A\x \le f(y)$ with $y \in I$.
\end{proof}

Lemma~\ref{lem:one_segment} can be easily  boosted to a $k$-dimensional parameter space $W$ for a fixed $k$:
\begin{equation}\label{eq:k_dim_param}
W \; = \; \{f(\y) \in \rr^{m} \; : \; \y \in R\}
\end{equation}
with $f : \rr^{k} \to \rr^{m}$ an RA and $R \subset \rr^{k}$ a rectangular box.

\begin{lem}\label{lem:one_piece}
Assume~\eqref{eq:num_pieces} holds.
Given a $\PIP$ problem 
with a $k$-dimensional parameter space $W$ \eqref{eq:k_dim_param},
there exists a finite set $\T = \{(S_{j},T_{j})\}$  of pairs of RAs $\ts T_{j} : \R^{k} \to \R^{n}$ and $S_{j} : \Z^{n} \to \Z^{n}$ so that the following hold.
For every $\y \in R \cap \zz^{k}$ and $\ts\b = f(\y) \in \rr^{m}$, we have:
\begin{equation*}
K_{\b} \cap \Z^{n} \neq \varnothing \quad \iff \quad \ex (S_{j},T_{j}) \in \T \; : \; S_{j} \floor{ T_{j} \y } \in K_{\b}\ts.
\end{equation*}
Furthermore, the set $\T$ contains at most $c(n,k)$ pairs $(S_{j},T_{j})$,
where $c(n,k)$ is a constant which depends only on $n$ and~$k$.
\end{lem}

\begin{proof}
WLOG, assume $R = [0,r_{1}) \times \ldots \times [0,r_{k})$.
We ``flatten'' the $k$-dimensional parameter~$\y$.
For every $\y = (y_{1},\dots,y_{k}) \in R$, let:
\begin{equation}\label{eq:flatten}
y' \; = \; y_{1} + y_{2}r_{1} + y_{3}(r_{1}r_{2}) + \ldots + y_{k}(r_{1} \cdots r_{k-1}) \; \in \; [0,\, r_{1}\cdots r_{k}).
\end{equation}
This RA maps the integer points in $R$ bijectively to those in $I = [0,\, r_{1} \cdots r_{k})$.
We rewrite $A\x \le f(\y)$ as another PIP problem with a $1$-dimensional parameter $y' \in I$ and $n+k$ variables $\x \in \zz^{n}$, $\y \in \zz^{k}$:
\begin{equation}\label{eq:flattened}
\gathered
y' \; = \; y_{1} + y_{2}r_{1} + y_{3}(r_{1}r_{2}) + \ldots + y_{k}(r_{1} \cdots r_{k-1}),\\
0 \le y_{i} < r_{i} \;\; \text{for} \;\; 1 \le i \le k, \quad A\x - f(\y) \le 0.
\endgathered
\end{equation}

Note that~\eqref{eq:flattened} has a solution if and only if the original PIP problem $A\x \le f(\y)$ has a solution.
Furthermore, in every solution of~\eqref{eq:flattened}, the variables $\y$ are uniquely determined by $y'$ via the RA~\eqref{eq:flatten}.
Applying Lemma~\ref{lem:one_segment}, we get a set $\T' = \{(S'_{j},T'_{j})\}$ of at most $c(n,k) \coloneqq (n+k+2)^{4(n+k+2)}$ candidates for~\eqref{eq:flattened}, where $T'_{j} : \R^{1} \to \R^{n+k}$ and $S'_{j} : \zz^{n+k} \to \zz^{n+k}$ are pairs of RAs.
Using~\eqref{eq:flatten}, we can re-express each pair $(S'_{j},T'_{j})$ as a pair $(S_{j},T_{j})$ with $T_{j} : \rr^{k} \to \rr^{n}$ and $S_{j} : \zz^{n} \to \zz^{n}$ so that~\eqref{eq:flattened} has a solution if and only if $\x = S_{j} \floor{T_{j} \y}$ satisfies $A\x \le f(\y)$ for some $j$.
In other words, $\cal T = \{(S_{j}, T_{j})\}$ is a finite set of at most $c(n,k)$ candidates for the original PIP problem $A\x \le f(\y)$.
\end{proof}

\begin{rem}\label{rem:floor}{\rm
Since the dimensions of $A$ are fixed, each condition $S_{ij} \floor{ T_{ij} \y } \in K_{\b}$ can be expressed as a short Boolean combination of linear inequalities, at the cost of introducing a few extra $\ex$ or $\for$ quantifiers. For example, a condition $\frac{1}{2} + \floor{y/5} \le 3$ for $y \in \zz$ can be expressed as either
\begin{equation}\label{eq:floor}
\exists\ts t
\begin{Bmatrix}\,
t &\le &y/5\\
t &> &y/5 - 1\\
\, \frac{1}{2} + t &\le &3\,
\end{Bmatrix}
\quad \text{or} \quad
\for t
\begin{bmatrix} \,
t &> &y/5\\
t &\le &y/5 - 1\\
\, \frac{1}{2} + t &\le &3 \,
\end{bmatrix}.
\end{equation}
Here $\{\cdot\}$ is a conjunction and $\left[\cdot\right]$ is a disjunction.
}\end{rem}

Now we relax the parameter space $W$ to an arbitrary $k$-dimensional polyhedron, i.e.,
\begin{equation}\label{eq:k_dim_polytope}
W \; = \; \{f(\y) \in \rr^{m} \; : \; \y \in Q\}
\end{equation}
with $f : \rr^{k} \to \rr^{m}$ an RA and $Q \subset \rr^{k}$ a polyhedron.

\begin{cor}\label{cor:mid_point}
Assume~\eqref{eq:num_pieces} holds.
Then for every fixed $m,n$ and $k$, there is a constant $d(m,n,k)$ so that the following holds.
For a $\PIP$ problem 
with a $k$-dimensional parameter space $W$ \eqref{eq:k_dim_polytope},
let:
\begin{equation*}
Q' \; = \; \big\{\ts \y \in Q \cap \zz^{k} \; : \; A\x \le f(\y) \;\; \text{has no solutions} \;\; \x \in \zz^{n} \ts \big\}.
\end{equation*}
If $|Q'| > d(m,n,k)$, then it contains three distinct points $\y_{1},\y_{2},\y_{3}$ with $\y_{3} = (\y_{1} + \y_{2})/2$.
\end{cor}
\begin{proof}
Let $R$ be a large enough box that contains $Q$.
Applying Lemma~\ref{lem:one_piece} to the PIP problem $A\x \le f(\y)$ with $\y \in R$, we get a set of candidates $\T = \{(S_{j}, T_{j})\}$ of size at most $c(n,k)$ so that:
\begin{equation*}
A\x \le f(\y) \;\; \text{has no solutions} \quad \iff \quad \for (S_{j},T_{j}) \in \T \; : \; S_{j} \floor{ T_{j} \y } \not\le f(\y).
\end{equation*}

By the argument in Remark~\ref{rem:floor}, each condition $ S_{j} \floor{ T_{j} \y } \not\le f(\y)$ can be expressed by a short Presburger formula $\ex \t \; \Phi_{j}(\y,\t)$ with length bounded in $m$ (fixed).
Taking conjunction over all such formulas for $1 \le j \le c(n,k)$, we have:
\begin{equation}\label{eq:criterion}
A\x \le f(\y) \;\; \text{has no solutions} \quad \iff \quad \ex\wt\t \quad \Phi(\y,\wt\t).\footnote{Separate variables $\t$ for different $\Phi_{j}$ must be concatenated into $\wt\t$.}
\end{equation}
Here $\Phi$ is still a short Presburger expression in a bounded number of variables.
Denote by $\lambda$ and $\mu$ the total number of variables and inequalities in $\Phi$, respectively.
Both of these are constants in $m,n$ and $k$.
Let $d = d(m,n,k) = 2^{\lambda+\mu}$.
The $\mu$ inequalities in $\Phi$ determine $\mu$ hyperplanes in $\rr^{\lambda}$.
These hyperplanes partition $\rr^{\lambda}$ into polyhedral regions:
\begin{equation*}
\rr^{\lambda} \; = \; W_{1} \sqcup \dots \sqcup W_{\eta},
\end{equation*}
with $\eta \le 2^{\mu}$.
Observe that as $(\y,\wt\t)$ varies over a single region $W_{j}$, the value of $\Phi(\y,\wt\t)$ is always true or always false.
Since $|Q'| > d$, we have at least $d+1$ distinct pairs $(\y_{1},\wt\t_{1}), \dots, (\y_{d+1},\wt\t_{d+1})$ for each of which $\Phi(\y_{i},\wt\t_{i}) = \text{true}$.
By the pigeon hole principle, some region $W_{j}$ contains at least $2^{\lambda}+1$ of these pairs.
Each such pair is a point in $\zz^{\lambda}$, so at least two of them must have coordinates equal mod $2$ pairwise.
Assume $(\y_{1},\wt\t_{1})$ and $(\y_{2},\wt\t_{2})$ are two such two pairs.
By convexity, $(\y_{1}+\y_{2},\wt\t_{1}+\wt\t_{2})/2$ is another integer point in $W_{j}$.
Since $\Phi$ is always true over $W_{j}$, this pair also satisfies $\Phi$.
By~\eqref{eq:criterion}, the point $\y_{3}=(\y_{1}+\y_{2})/2$ also lies in $Q'$.
We conclude that $\y_{1},\y_{2},\y_{3} \in Q'$.
\end{proof}

\medskip

\begin{theo}\label{t:KPT_wrong}
The bound~\eqref{eq:num_pieces} as claimed by $\KPT$ does not hold in full generality.
In other words, even for $k=1$ and fixed $m,n$,
the number of pieces $r$ in the partition~\eqref{eq:KPT_partition} must be at least
$\exp(\ep\phi)$ for some constant $\ep=\ep(m,n)>0$.
\end{theo}

\begin{proof}
Assume~\eqref{eq:num_pieces} holds.
Consider the following continued fraction of length $(2\vk+1)$:
\begin{equation*}
\al_{\vk} \; = \; [2;\; 1,\, \dots,\, 1] \; = \; p/q\ts,
\end{equation*}
where $p = F_{2\vk+3},\, q = F_{2\vk+1}$ are the Fibonacci numbers.
From Properties (G1)--(G6) in Section~\ref{sec:frac}, we see that the
lower convex curve $\mathcal{C}$ for $\al$ connects $\vk+2$ integer points:
\begin{equation*}
C_{0} = (0,1),\; C_{1} = (2,1),\; C_{2} = (5,2) ,\;  \dots ,\; C_{\vk+1} = (p,q).\footnote{Recall that the vertical coordinate is put in the first position.}
\end{equation*}
Here $C_i=(F_{2i+1},F_{2i-1})$ for $1 \le i \le s+1$.
Let $\mathcal{C'}$ be the convex curve connecting $C_{1},\dots,C_{\vk+1}$ (see Figure~\ref{f:continued_fraction}).
Property (G2), for every $1 \le i \le \vk$, the segment $C_{i}C_{i+1}$ has exactly $2$ integer points,
$C_{i}$ and $C_{i+1}$.
In other words, we have $\mathcal{C'} \cap \zz^{2} = \{C_{1},\dots,C_{\vk+1}\}$.

Let $Q$ be the triangle defined in~\eqref{eq:Q_def}.
By Lemma~\ref{lem:easy}, an integer point $\y = (y_{2},y_{1}) \in Q$ lies on
$\mathcal{C'}$ if and only if $P_{\y}$ is integer point free, where $P_{\y}$
was defined in~\eqref{eq:P_y_def}.\footnote{We take the first term in~$\al$ to be~$2$
because of Remark~\ref{rem:exception}}
In other words, we have:
\begin{equation*}
\aligned
Q' \; &= \; \left\{ \y \in Q \cap \zz^{2} \; : \;
\left\{
\begin{matrix}
p y_{1}  -  q y_{2} &\ge &p x_{1} - q x_{2} &\ge &0\\
y_{2} - 1 &\ge &x_{2} &\ge &1\\
\end{matrix}
\right\}
\; \text{has no solutions} \; (x_{2},x_{1}) \in \zz^{2}
\right\}
\\
&= \; \mathcal{C'} \cap \zz^{2}\ts.
\endaligned
\end{equation*}

The above is a PIP problem with parameters $\y \in Q$ and variables $\x = (x_{2},x_{1}) \in \zz^{2}$.
Note that the system has fixed length $m=4$.
By Corollary~\ref{cor:mid_point}, there exists a constant~$d$, so that if
$|\mathcal{C'} \cap \zz^{2}| = \vk+1 > d$ then there are $3$ distinct points
$\y_{1},\y_{2},\y_{3} \in \mathcal{C'} \cap \zz^{2}$ with $\y_{3} = (\y_{1} + \y_{2})/2$.
However, by the previous paragraph, the only integer points on $\mathcal{C'}$
are $C_{1},\dots,C_{\vk+1}$, which are in convex position, see Property~(G4).
Thus, none among them can be the midpoint of two others.
We get a contradiction.
Therefore,~\eqref{eq:num_pieces} cannot hold in general.

Recall the PIP problem~\eqref{eq:add_var} with a $1$-dimensional parameter $y'$, i.e., $k=1$.
From~\eqref{eq:num_pieces}, we deduced $r \ll M$ in~\eqref{eq:r_M}.
This led to the observation that at least one interval $I'$ must lie in a single piece $J_{i}$.
The chain of deductions continued from there through Lemma~\ref{lem:one_piece} and Corollary~\ref{cor:mid_point} and led to the above contradiction.
Therefore, we must have $r > M$, which implies $r \ge 2^{\ep\phi}$ for some constant $\ep=\ep(m,n)>0$.
\end{proof}

\subsection{Implications}\label{ss:kannan-ES}
To summarize, Theorem~\ref{t:KPT_wrong} shows that a polynomial size
decomposition into polyhedral pieces as in~\eqref{eq:KPT_partition} does not exist.
If one is willing to sacrifice the polyhedral structure of the pieces,
then a polynomial size partition similar to~\eqref{eq:KPT_partition} does
in fact exist~\cite{ES} (see also~\cite{E2}):

\begin{theo}[Eisenbrand and Shmonin]\label{th:ES}
Fix $n$ and $k$. Let \ts $A\x \le \b$ \ts be a $\PIP$ problem with a $k$-dimensional parameter space $W$.
Then we can find in polynomial time a partition
\begin{equation}\label{eq:ES_partition}
W \, = \, S_1 \, \sqcup \, S_2  \, \sqcup \, \dots \, \sqcup \, S_r\ts,
\end{equation}
where each $S_{i}$ is an \emph{integer projection} of another polyhedron $S'_{i} \subseteq \rr^{m+\ell}$, defined as:
\begin{equation*}
S_{i} \; = \; \big\{\ts \b \in \rr^{m} \; : \; \ex \t \in \zz^{\ell} \;\; (\b,\t) \in S'_{i} \ts\big\}.
\end{equation*}
Here $\ell = \ell(n)$ is a constant that depends only on~$n$.
All polyhedra $S'_{i}$ can be found in polynomial time. The partition~\eqref{eq:ES_partition}
satisfies all other properties as claimed in~$\KPT$.
\end{theo}

Note that the integer projection of a polyhedron defined in the theorem
is not necessarily a polyhedron as the following example shows.

\begin{eg}
Consider the polytope $S' = \big\{\ts (y_{2},y_{1}) \in \rr^{2} \; : \; 0 \le y_{2} \le 1,\; 0 \le y_{1} - 3y_{2} \le 2 \ts\big\}$.
The integer projection of $S'$ on the coordinate $y_{1}$ is $S = [0,2] \cup [3,4]$ (see Figure~\ref{f:int_proj}).
\end{eg}

\begin{figure}[hbt]
\begin{center}
\psfrag{y1}{$y_1$}
\psfrag{y2}{$y_2$}
\psfrag{O}{$O$}
\psfrag{1}{$1$}
\psfrag{2}{$2$}
\psfrag{3}{$3$}
\psfrag{4}{$4$}
\psfrag{S}{$S'$}

\epsfig{file=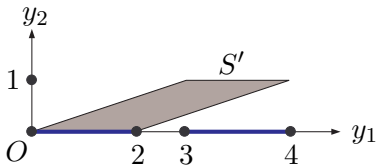, width=5.5cm}
\end{center}

\vspace{-1em}

\caption{A polytope $S'$ (shaded) and is integer projection (bold).}
\label{f:int_proj}
\end{figure}

We emphasize that the proofs of Theorem~\ref{th:Kannan} and Corollary~\ref{cor:PR2}
still hold if KPT is substituted by Theorem~\ref{th:ES} (see~\cite{ES}).
Overall, the only discrepancy between KPT and Theorem~\ref{th:ES}
is about the structures of the pieces in the partition.  This does not
at all affect all known results about decision with $2$ quantifiers or less.
Worth mentioning is the polynomial time algorithm by Barvinok and Woods~\cite{BW}
on counting integer points in the integer projection of a polytope.
This algorithm uses a weaker (valid) partitioning procedure also due
to Kannan~\cite[Lemma~3.1]{K2}.
However, as we pointed out in Section~\ref{sec:KPT_intro}, for $3$ quantifiers or more,
this structural discrepancy between KPT and Theorem~\ref{th:ES}
is of crucial importance.

\bigskip

\section{Final remarks and open problems}\label{sec:final}

\subsection{}\label{ss:finrem-KPT}
Niels Bohr, the inventor of quantum theory, is quoted saying:

\smallskip

\begin{itemize}
\item[] ``It is the hallmark of any deep truth that its negation is also a deep truth.''
\end{itemize}

\smallskip

\nin
This roughly reflects our attitude towards KPT.  A pioneer result at the time,
it only slightly overstated the truth compared to the Eisenbrand--Shmonin theorem
(Theorem~\ref{th:ES}).   In fact, for many applications, including Kannan's
Theorem~\ref{th:Kannan} and Barvinok--Woods algorithm \cite{BW},
Kannan's weaker result in~\cite{K2} is sufficient.

Let us emphasize that, of course, it would be natural to have a partition into
convex (co-)polyhedra rather than general semilinear sets, since convex
polyhedra are much easier to work with.  The fact that it took nearly 30 years
until KPT was disproved, shows both the delicacy and the technical
difficulty of the issue.

\medskip

\subsection{}\label{ss:finrem-KPT-gap}
The gap in the proof of KPT (Theorem~3.1 in~\cite{K1}) could be traced to the following lines:

\smallskip

\begin{itemize}
\item[]``\ldots for each $\ts (b,x) \in S_{i}$ \ts (with $b \in P$, $x \in \zz^{n}$), there is a unique $\ts y \in \zz^{\ell}$ \ts  so that $(b,x,y)$ belongs to $S_{i}'$. In fact, each component of $y$ is of the form $F'\floor{Fx}$, where $F',F$ are affine transformations. This is easily proved by induction on $n$, noting that (4.5) of [8], the $z$ is in fact forced to be $\floor{\al+1-\be}$.''
\end{itemize}

\smallskip

\nin
Here [8] refers to the conference proceedings version of paper~\cite{K2}.
In equation (4.5) of~\cite{K2}, variable~$z$ is in fact forced to be $\floor{\al+1-\be}$.
However, the quantity $\al$ in (4.5) actually depends on $b$, which makes $\floor{\al+1-\be}$ a function of $b$ instead of a constant.
This implies that $y$ in the above quoted paragraph could also depend on $b$.
This technical error was perhaps due to the unclear notation $\al$, which does not reflect its dependence on $b$, or due to the complicated cross referencing between~\cite{K1} and~\cite{K2}.

\medskip

\subsection{}
There is a delicate difference between the treatment of $\PIPP$ in Section~\ref{ss:kannan-KPT} versus that in the integer programming literature (see e.g.~\cite{CL,V+,VW}). In the latter, the
parameter space~$W$ is also partitioned into convex polyhedra $P_{i}$,
and over each $P_{i}$ the number of solutions $\x$ is given by a
quasi-polynomial $p_{i}(\b)$ in $\b$.  However,
since there are no test sets, this does not allow us to solve $\PIPP$ for \emph{all} $\b$.
In other words, even though a quasi-polynomial $p_{i}({\b})$ is obtained, which evaluates to $|K_{\b} \cap \zz^{n}|$, there is no easy way to test whether $p_{i}(\b) \neq 0$ for all $\b$ within $P_{i}$.
In general, we prove in~\cite{shortGF} that there are strong
obstacles in using (short) generating functions to decide feasibility of
Presburger sentences.

\medskip

\subsection{}
Now that we have Theorem~\ref{th:PR3hard}, one can ask if the dimension~$5$
is tight.
Observe that for three variables and three quantifiers, there is essentially a unique form of short Presburger
sentence:
$$
\exists \ts z  \;\; \forall \ts y \;\; \exists \ts x \, : \, \Phi(x,y,z).
$$
Despite Theorem~\ref{th:KPT-exp}, KPT actually holds for a PIP problem
$\ts a\ts x  \le f(y,z)\ts$ with a single variable~$x$, i.e., when $n=1$.
Therefore, this sentence can be decided by the approach in~\cite{shortPR}.
The only remaining special case of $\SPAt$ is
$$
\exists \ts z  \;\; \forall \ts y \;\; \exists \ts \x \, : \, \Phi(\x,y,z), \ \. \text{
where} \ \, \x \in \zz^2.
$$
It would be interesting to see if this case is also $\NP$-complete.

Similarly, for sentences $\GIP$, one can ask if dimension 6 in
Theorem~\ref{cor:system2} can be lowered.  We believe it can be,
at least for the counting part (cf.~\cite{KannanNPC}).

\medskip

\subsection{}
Motivated in part by the \emph{Hilbert's tenth problem},
Manders and Adleman~\cite{MA} (see also \cite[$\S$A7.2]{GJ})
proved the following classical result: \ts
feasibility over $\nn$ of
$$
a\ts x^2 \. + \. b \ts y  \.= \. c
$$
is $\NP$-complete, given $a,b,c\in \zz$.  One can view our
Theorem~\ref{cor:system1} as a related result, where a single
quadratic equation and two linear inequalities $x,y \ge 0$
(over~$\zz$) are replaced with a system of 24 linear inequalities.

\medskip

\subsection{}\label{ss:finrem-minmax}
Minimizing polynomial functions over integer points in a convex
polytope is an interesting problem of Integer Programming.  Already
for polynomials of degree~4 in two variables this is known to be
$\NP$-hard~\cite{DHKW}, but for lower degree polynomials some such
problems can be solved in polynomial time~\cite{DHWZ}.
The survey paper~\cite{Kop} contains extensive background on various related problems.
Curiously, the following natural problem remains open:

\begin{question}
Let $n$ be fixed. Given a polytope $P \subset \rr^{n}$ and a rational
quadratic function $f:\rr^{n} \to \rr$, can the optimization problem
$\min_{\x \in P \cap \zz^{n}} f(\x)$ be solved in polynomial time?
\end{question}

The case $n=2$ was resolved positively in~\cite{DeW}.
Note that the case $n=3$ with $f$ homogeneous is known to have an FPTAS~\cite{HWZ}.

\medskip

\subsection{}\label{ss:finrem-Pareto}
Our Theorem~\ref{th:Pareto} strongly contrasts with the positive results in~\cite{DHK},
which require that all $f_{i}$'s are linear.  There, it is proved that optimizing
over the Pareto minima can be done in polynomial time when $g$ is linear.
Furthermore, if $g$ is non-linear then an FPTAS also exists.
Here, we say that having even one $f_{i}$ quadratic is enough
to make the problem hard.

Note that in Theorem~\ref{th:Pareto} we use three polynomial functions,
two or which are linear.  It would be interesting to see if just two polynomial
functions suffice for the hardness.

%


\vskip.82cm


\subsection*{Acknowledgements}
We are greatly indebted to Sasha Barvinok for many fruitful discussions
and encouragement.  We are also grateful to Iskander Aliev,
Matthias Aschenbrenner, Art\"{e}m Chernikov,
Fritz Eisenbrand, Lenny Fukshansky,  Robert Hildebrand, Ravi Kannan,
Oleg Karpenkov, Matthias K\"oppe, Rafi Ostrovsky
and Kevin Woods for interesting conversations and helpful remarks.
Special thanks to Jes\'{u}s De Loera for suggesting hardness of
Pareto optima as a possible application of our main results.
This work was finished while both authors were in residence
of the MSRI long term Combinatorics program in the Fall of 2017;
we thank MSRI for the hospitality.  The first author
was partially supported by the UCLA Dissertation Year Fellowship.
The second author was partially supported by the~NSF.

\newpage

\end{document}